\documentclass[11pt,reqno]{amsart}

\newcommand{\cev}[1]{\reflectbox{\ensuremath{\vec{\reflectbox{\ensuremath{#1}}}}}}
\usepackage[utf8]{inputenc}
\usepackage{setspace}
\usepackage{geometry}
\usepackage{enumerate}
\usepackage{enumitem, xcolor, amssymb,latexsym,amsmath,bbm}
\usepackage{mathtools}
\usepackage[mathscr]{euscript}
\usepackage{amsmath}
\usepackage{amssymb}
\usepackage{enumitem} 
\newcommand{\R}{\mathbb{R}}
\newcommand{\C}{\mathbb{C}}

\newcommand{\M}{\mathrm{Mod}(S_g)}
\newcommand{\N}{\mathbb{N}}

\usepackage{tikz}
\usepackage{wrapfig}
\usepackage{enumerate}
\usepackage{graphicx}
\usepackage{subfigure}

\usetikzlibrary{arrows.meta}

\usetikzlibrary{decorations.markings}

\usepackage[colorlinks=true,citecolor=blue, linkcolor=blue,urlcolor=blue]{hyperref}

\calclayout

\theoremstyle{plain}
\newtheorem{theorem}{Theorem}[section]
\newtheorem{corollary}[theorem]{Corollary}
\newtheorem{lemma}[theorem]{Lemma}
\newtheorem{proposition}[theorem]{Proposition}
\theoremstyle{definition}
\newtheorem{definition}[theorem]{Definition}
\newtheorem{remark}[theorem]{Remark}

\newtheorem{question}[theorem]{Question}

\newtheorem{example}[theorem]{Example}

\numberwithin{equation}{section}
\numberwithin{figure}{section}
\allowdisplaybreaks

\title{Filling with separating curves}
\date{\today}
\author{Bhola Nath Saha}
\address{
Department of Mathematics and Statistics\\ 
Indian Institute of Technology  \\ 
Kanpur, Uttar Pradesh-208016\\
India}
\email{bnsaha@iitk.ac.in}
\author{Bidyut Sanki}
\address{
Department of Mathematics and Statistics\\ 
Indian Institute of Technology  \\ 
Kanpur, Uttar Pradesh-208016\\
India}
\email{bidyut@iitk.ac.in}

\begin{document}


\maketitle

\begin{abstract}
   A pair $(\alpha, \beta)$ of simple closed curves on a closed and orientable surface $S_g$ of genus $g$ is called a filling pair if the complement is a disjoint union of topological disks. If $\alpha$ is separating, then we call it as separating filling pair. In this article, we find a necessary and sufficient condition for the existence of a separating filling pair on $S_g$ with exactly two complementary disks. We study the combinatorics of the action of the mapping class group $\M$ on the set of such filling pairs. Furthermore, we construct a Morse function $\mathcal{F}_g$ on the moduli space $\mathcal{M}_g$ which, for a given hyperbolic surface $X$, outputs the length of the shortest such filling pair with respect to the metric in $X$. We show that the cardinality of the set of global minima of the function $\mathcal{F}_g$ is the same as the number of $\M$-orbits of such filling pairs. 
\end{abstract}
\section{Introduction}
Let $S_g$ be a closed orientable surface of genus $g$. A \emph{simple closed curve} $\alpha$ on $S_g$ is a continuous injective map $\alpha:S^1\rightarrow S_g$, where $S^1$ is the unit circle. For two simple closed curves $\alpha$ and $\beta$, the \emph{geometric intersection number} between them, denoted by $i(\alpha,\beta)$, is defined by $$i(\alpha,\beta) = \min\limits_{\alpha'\sim \alpha} |\alpha'\cap \beta|,$$ where $\sim$ stands for the free homotopy relation. Two simple closed curves $\alpha$ and $\beta$ are said to be in \emph{minimal position} if their geometric intersection number satisfies $i(\alpha,\beta)=|\alpha\cap\beta|$ (for more details, we refer the readers to~\cite{farb2011primer}, Section 1.2.3). 

A set $\Omega=\{\alpha_1,\hdots,\alpha_n\}$ of simple closed curves on $S_g$ is called a \emph{filling set} if for $i\neq j$, the curves $\alpha_i$ and $\alpha_j$ are in minimal position  and $S_g\setminus(\cup_{i=1}^{n}\alpha_i)$ is a disjoint union of topological disks. If the number of curves in a filling set is $n=2$, then it is called as \emph{filling pair} and if the number of disks in $S_g\setminus(\cup_{i=1}^{n}\alpha_i)$ is minimal or equivalently the total number of intersection points between the curves is minimal, then the filling system is called minimally intersecting. 

The study of filling pairs has its importance for its various applications. For example, Thurston has used filling sets of size two to construct the most nontrivial mapping class group elements, so-called pseudo-Anosov homeomorphisms (for details see~\cite{farb2011primer}, Section 13.2.3). More recently, in~\cite{penner1988construction}, Penner has extended Thurston's construction to pairs of filling multi-curves as follows: given multi-curves $A= \{ a_1, \dots, a_m \}$ and $B=\{b_1,\dots,b_n \}$ such that $A\cup B$ fills the surface S. Then a homeomorphism $f$ which is the product of $T_{a_i}$ and  $T_{b_j}^{-1}$, where each $a_i$ and $b_j$ occur at least once, is a pseudo-Anosov homeomorphism. Note that $T_{a_i}$ denotes the Dehn twist about the simple closed curve $a_i$. Another motivation for studying filling sets is in the problem of the construction of a spine for the moduli space of $S_g$. In \cite{thurstonspine}, Thurston has proposed $\chi_g$, the set of all closed hyperbolic surfaces of genus $g$ having a filling set of systolic geodesics, as a candidate spine of moduli space of $S_g$. He has given a sketch of a proof but unfortunately, that appears difficult to understand. Recently, Bourque \cite{fortier2023dimension} has studied the dimension of Thurston's spine $\chi_g$. The author has shown that for every $\epsilon >0$, there exists an integer $g \geq 2$ such that $\chi_g$ has dimension at least $(5 - \epsilon)g$ (see Theorem 1.1 \cite{fortier2023dimension}). More recently, Mathieu \cite{mathieu2023estimating} has shown the existence of an infinite set $A$ of integers $g \geq 2$ such that $\mathrm{codim}(\chi_g)\in o(g/\sqrt{\mathrm{log}(g)}$, when $g\in A$ goes to infinity.

Fillings of surfaces have been studied extensively by Aougab-Huang \cite{aougab2015minimally}, Sanki \cite{sanki2018filling}, Jeffreys \cite{jeffreys2019minimally}, Aougab-Menasco-Nieland \cite{aougab2022origamis}  and many others \cite{anderson2011small}, \cite{MR3877282}. In \cite{aougab2015minimally}, the authors have shown the existence of minimally intersecting filling pairs on $S_g$ for $g\geq3$ and estimated the bounds of the number of $\M$-orbits of such fillings. In \cite{sanki2018filling}, the author has generalized the result in \cite{aougab2015minimally}, where he has shown the existence of filling pairs on $S_g$ for $g\geq2$ with any given number of complementary regions. Luke Jeffreys further extended the study of minimally intersecting filling pairs on punctured surfaces.  In \cite{jeffreys2019minimally}, he has shown the existence of minimally intersecting filling pairs on punctured surfaces. Furthermore, in \cite{aougab2022origamis}, the authors have considered those minimally intersecting filling pairs for which the geometric intersection number and algebraic intersection number are the same and they find a lower bound for the number of $\M$-orbits of such pairs.

The simple closed curves on a surface are classified into two classes: separating and non-separating. A simple closed curve $\alpha$ on $S_g$ is called \emph{separating} if $S_g\setminus\alpha$, the complement of $\alpha$ in $S_g$, is disconnected. Otherwise, it is called \emph{non-separating}. Mirzakhani \cite{MirzakhaniGrowth}, Delecroix-Goujard-Zograf-Zorich \cite{Delecroix} studied the frequency of separating and nonseparating simple closed geodesics on closed hyperbolic surfaces. The construction of filling pairs with one or two separating curves on hyperelliptic surfaces can be found in the work of Jeffreys \cite{jeffreys2022meanders}. In this paper, we study the topology, geometry and combinatorics of filling pairs containing at least one separating curve. A filling pair $(\alpha,\beta)$, where at least one of $\alpha$ and $\beta$ is a separating simple closed curve, is called a \emph{separating filling pair}. We will always assume that $\alpha$ is a separating curve in a separating filling pair $(\alpha,\beta)$.

Suppose $(\alpha,\beta)$ is a filling pair on $S_g$. We have the following natural question.
\begin{question}\label{q:1}
How to determine whether a filling pair $(\alpha,\beta)$ is a separating filling pair?
\end{question}
We note that a minimally intersecting separating filling pair has exactly one separating curve. In our first result, we answer Question~\ref{q:1}. In particular, we prove a characterisation theorem which gives a necessary and sufficient condition for a curve in a filling pair to be separating. The key ingredient used in the proof is topological graph theory.  

Given a filling pair $(\alpha,\beta)$, we associate a number $b$, the number of topological disks in the complement. A simple Euler's characteristic argument implies that
$i(\alpha,\beta)=2g-2+b.$
If $(\alpha,\beta)$ is a separating filling pair, then we have $b\geq2$ and hence $i(\alpha,\beta)\geq2g$. If $i(\alpha,\beta)=2g$, we call it as a \emph{minimally intersecting separating filling pair} or simply a \emph{minimal separating filling pair}. In~\cite{aougab2015minimally}, Aougab and Huang have shown that for $g\geq 3$, there exists a filling pair on $S_g$ with exactly one connected component in the complement. We note that the curves in such filling pairs are necessarily non-separating. In the context of minimal separating filling pairs, we study an analogous question. We prove the existence of minimal separating filling pairs in the following theorem. 

\begin{theorem}\label{theorem:2}
There exists a minimally intersecting separating filling pair on $S_g$ if and only if g is even and $g\geq4$.
\end{theorem}

The proof of Theorem~\ref{theorem:2} is constructive. We explicitly construct such a filling pair on $S_4$ and for the existence on $S_g, g\geq 6$ even, we use the method of induction on $g$. The converse part is proved using the standard Euler's characteristic formula.

For $g\geq 3$ odd, any separating filling pair $(\alpha,\beta)$ on $S_g$ has at least four complementary disks or equivalently $i(\alpha,\beta)\geq (2g+2)$. We prove the theorem below.
\begin{theorem}\label{thm:new}
    For genus $g\geq 3$ odd, there exists a separating filling pair $(\alpha_g, \beta_g)$ on $S_g$ with four complementary disks or equivalently $i(\alpha_g,\beta_g)= (2g+2)$.
\end{theorem}

The mapping class group $\M$ of a surface $S_g$ of genus $g$ is the group of all orientation preserving self-homeomorphisms up to isotopy  (\cite{farb2011primer}, Chapter 2). The group $\M$ acts on the set of minimally intersecting separating filling pairs as follows: for a minimally intersecting separating filling pair $(\alpha,\beta)$ and $f\in \M$, $f\cdot(\alpha,\beta)=(f(\alpha),f(\beta))$. We estimate the number of $\M$-orbits of this action in the following theorem.
\begin{theorem}\label{thm:1.3}
If $N(g)$ is the number of $\mathrm{Mod}(S_g)$-orbits of minimally intersecting  separating filling pairs, then 
\begin{align*}
\frac{\prod\limits_{k=1}^{\frac{g-4}{2}}(3k+5)}{4\times (2g)^2\times (\frac{g-4}{2})!}\leq N(g) \leq 2(2g-2)(2g-2)!.
\end{align*}
\end{theorem}
The moduli space $\mathcal{M}_g$ of genus $g\geq 2$ is the collection of all hyperbolic metrics on $S_g$ up to isometry. As in \cite{aougab2015minimally}, we construct a topological Morse function on $\mathcal{M}_g$ using the length function of a filling pair. Consider the set, $$ \mathcal{C}_g=\{(\alpha,\beta): (\alpha,\beta)\text{ is a minimally intersecting  separating filling pair on } S_g\}.$$
For $X\in \mathcal{M}_g$, the length $l_X(\alpha,\beta)$ of $(\alpha,\beta)\in \mathcal{C}_g$ is defined as $l_X(\alpha,\beta)=l_X(\alpha)+l_X(\beta),$ where $l_X(\alpha)$ is the length of the geodesic in the free homotopy class of $\alpha$ with respect to the hyperbolic metric X. Now, we define a function $\mathcal{F}_g :\mathcal{M}_g \longrightarrow \mathbb{R}$, by $$
    \mathcal{F}_g(X)=\min \{ l_X(\alpha,\beta) | (\alpha,\beta) \in \mathcal{C}_g\}.$$
We show that, 
\begin{theorem}\label{theorem_lenth_function}
For $g\geq4$, $\mathcal{F}_g$ is a proper and topological Morse function. For any $X\in \mathcal{M}_g$, we have 
\begin{center}
    $\mathcal{F}_g\geq m_g,$
\end{center}
where 
\begin{align*}
m_g =4g\times \cosh ^{-1}\left( 2\left[\cos \left(\frac{2\pi}{4g}\right)+\frac{1}{2}\right]\right)
\end{align*}
denotes the perimeter of a regular right-angled hyperbolic $4g$-gon.\\
Suppose $\mathcal{B}_g=\{X\in\mathcal{M}_g:\mathcal{F}_g(X)=m_g\}$. Then $\mathcal{B}_g$ is a finite set and $|\mathcal{B}_g|=N(g)$. If $X\in \mathcal{B}_g$, the injectivity radius $r_{\mathrm{inj}}(X)$ of $X$ satisfies
$$r_{\mathrm{inj}}(X)\geq\frac{1}{2}\times \cosh^{-1}\left[ 8\cos^3\left(\frac{2\pi}{4g}\right)+8\cos^2\left(\frac{2\pi}{4g}\right)-1\right].$$
\end{theorem}
\section{Fat Graphs}

In this section, we recall some definitions and notations on fat graphs which are essential in the subsequent sections. We begin with recalling the definition of a graph. The definition of graph used here is not the standard one, but it is straightforward to see that it is equivalent to the standard one. Such a definition is convenient for defining fat graphs.
\begin{definition}
A graph $G$ is a triple $G=(E,\sim,\sigma_1)$, where
\begin{enumerate}
    \item $E=\{\vec{e}_1,\cev{e}_1,\dots,\vec{e}_n,\cev{e}_n\}$ is a finite, non-empty set, called the set of directed edges.
    \item $\sim$ is an equivalence relation on $E$.
    \item $\sigma_1:E\to E$, a fixed point free involution, maps a directed edge to its reverse directed edge, i.e., $\sigma_1(\vec{e_i})=\cev{e_i},$ for all $\vec{e_i}\in E$.
\end{enumerate}
\end{definition}
 
The equivalence relation $\sim$ is defined by $\vec{e}_1\sim\vec{e}_2$ if $\vec{e}_1$ and $\vec{e}_2$ have same initial vertex. In an ordinary language, $V=E/\sim$ is the set of all vertices and $E/\sigma_1$ is the set of un-directed edges of the graph. The number of edges incident at a vertex is called the \textit{degree} of the vertex (for more details, we refer to~\cite{sanki2018filling}, Section 2). Now, we define fat graphs in the following.
\begin{definition}
A \textit{fat graph} is a quadruple $\Gamma=(E,\sim,\sigma_1,\sigma_0)$, where
\begin{enumerate}
    \item $G=(E,\sim,\sigma_1)$ is a graph.
    \item $\sigma_0$ is a permutation on E so that each cycle corresponds to a cyclic order on the set of oriented edges going out from a vertex.
\end{enumerate}
\end{definition}
\begin{definition}
    A fat graph is called \textit{decorated} if the degree of each vertex is even and at least $4$.
\end{definition}
\begin{definition}
    A simple cycle in a decorated fat graph is called a \textit{standard cycle} if every two consecutive edges in the cycle are opposite to each other in the cyclic order on the set of edges incident at their common vertex. If a cycle is not standard,
    we call it as non-standard.
\end{definition}
\textbf{Surface associated to a fat graph.}
Given a fat graph $\Gamma=(E, \sim, \sigma_1, \sigma_0)$, we construct an oriented topological surface $\Sigma(\Gamma)$ with boundary as follows: Consider a closed disk corresponding to each vertex and a rectangle corresponding to each un-directed edge. Then identify the sides of the rectangles with the boundary of the disks according to the order of the edges incident to the vertex. See Figure~\ref{loc_pic_of_fat_graph} for a local picture. The boundary of a fat graph $\Gamma$ is defined as the boundary of the surface $\Sigma(\Gamma).$
\tikzset{->-/.style={decoration={
  markings,
  mark=at position #1 with {\arrow{>}}},postaction={decorate}}}
  
  \tikzset{-<-/.style={decoration={
  markings,
  mark=at position #1 with {\arrow{<}}},postaction={decorate}}}

\begin{figure}[htbp]
\begin{center}
\begin{tikzpicture}[xscale=1,yscale=1]
\draw (-1.5,0)--(1.5,0);
\draw (0,1.5)--(0,-1.5);
\draw[->-=1] (0.6062,.35) arc (30:240:.7);
\draw (5,0) circle (.7cm);
\draw (3,.2)--(4.34,.2);
\draw (3,-.2)--(4.34,-.2);
\draw (5.66,.2)--(7,.2);
\draw (5.66,-.2)--(7,-.2);
\draw (4.8,.66)--(4.8,2);
\draw (5.2,.66)--(5.2,2);
\draw (4.8,-.66)--(4.8,-2);
\draw (5.2,-.66)--(5.2,-2);
\end{tikzpicture}
\end{center}
 \caption{Local picture of the surface obtained from a fat graph.} 
\label{loc_pic_of_fat_graph}
\end{figure}
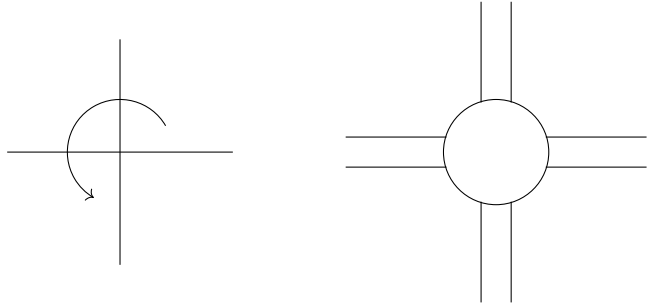

\begin{example}
Let us consider the fat graph $\Gamma=(E,\sim,\sigma_1,\sigma_0)$, where
\begin{enumerate}[label=(\roman*)]
    \item $E=\{\vec{e}_i,\cev{e}_j|i,j=1,2,3\}$, is the set of all directed edges and $E_1=\{e_1, e_2, e_3\}$ is the set of all un-directed edges, where $e_i=\{\vec{e}_i, \cev{e}_i\}, i=1,2,3.$
    \item $V=\{v_1, v_2\}$, where $v_1=\{\vec{e}_1,\vec{e}_2,\vec{e}_3\}$ and $v_2=\{\cev{e}_1,\cev{e}_2,\cev{e}_3\}$, is the set of vertices.
    \item $\sigma_1(\vec{e_i})=\cev{e_i}$, $i=1,2,3,$ and
    \item $\sigma_0=(\vec{e}_1,\vec{e}_3,\vec{e}_2)(\cev{e}_1,\cev{e}_3,\cev{e}_2)$.
\end{enumerate}  

The fat graph $\Gamma$ and the associated surface $\Sigma(\Gamma)$ are given in Figure~\ref{fat_graph_example_two_vertices}. A simple Euler characteristic argument and topological classification of surfaces imply that $\Sigma(\Gamma)$ is homeomorphic to a compact surface of genus one and with a single boundary component. We also observe that if the boundary of the surface is labelled with the convention as in Figure \ref{fat_graph_example_two_vertices}, the boundary corresponds to the cycles of $\sigma_0^{-1}\sigma_1$ which is $(\vec{e}_1,\cev{e}_2,\vec{e}_3,\cev{e}_1,\vec{e}_2,\cev{e}_3)$. We generalize this observation in Lemma \ref{lemma:2.6}.

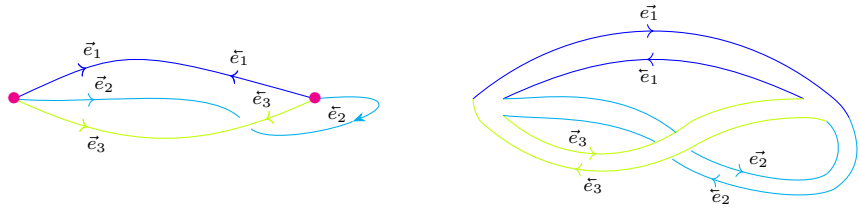
\begin{figure}[htbp]
\begin{center}
\begin{tikzpicture}[xscale=1, yscale=1]
\draw[cyan][->-=.35] (0,0)..controls(1,-0.1)and(2.5,.2)..(3,-.25);
\draw[cyan][-{Stealth[color=cyan]}] (4,0)..controls(5,.1)and(5,-.1)..(4.5,-.3);
\draw[cyan] (3.15,-.4)..controls(3.4,-.6)and(4.3,-.4)..(4.5,-.3);
\draw[blue, ->-=.25, -<-=.74] (0,0)..controls(1.5,.7)and(1.5,.7)..(4,0);

\draw[lime, ->-=.25, -<-=.85] (0,0)..controls(1.5,-.7)and(2.5,-.7)..(4,0);
\draw[magenta] (0,0) node{$\bullet$} (4,0) node{$\bullet$};
\draw[blue, -<-=.5] (6.5,0)..controls(8,.7)and(9,.7)..(10.5,0);
\draw[blue, ->-=.5] (6.1,0)to[bend left=40](10.9,0);
\draw[lime, ->-=.5] (6.5,-.22) to [bend right=40] (9,-.3);
\draw[cyan](6.5,0)..controls(7,0)and(8,.2)..(8.8,-.45);
\draw[cyan](6.5,-.22)..controls(7,-.28)and(7.8,-.2)..(8.52,-.57);
\draw[cyan, ->-=.4](9,-.69)..controls(9.5,-1)and(10.5,-1.2)..(10.8,-1);
\draw[cyan](10.8,-1)to [bend right=50] (10.8,-.3);
\draw[cyan, -<-=.3](8.75,-.8)..controls(9.3,-1.2)and(10.5,-1.4)..(10.9,-1.2);
\draw[cyan](10.9,-1.2)..controls(11.2,-1)and(11.3,-.7)..(11.1,-.25);
\draw[lime, -<-=.5] (6.2,-.3) to [bend right=35] (9,-.65);
\draw[lime](9,-.65)..controls(9.5,-.35)and(10,-.25)..(10.5,-.25);
\draw[lime] (9,-.3)..controls(9.5,-0.05)and(10,0)..(10.5,0);
\draw[lime] (6.1,0) to[bend right=15](6.2,-.3);
\draw[lime] (10.8,-.3) to[bend right=10](10.5,-.25);
\draw[blue](10.9,0)to[bend left=10](11.1,-.25);

\draw (1.05,.65) node{\tiny$\vec{e}_1$};
\draw (1.2,.2) node{\tiny$\vec{e}_2$};
\draw (1.1,-.6) node{\tiny$\vec{e}_3$};
\draw (3,.55) node {\tiny{$\cev{e_1}$}} (3.3,0) node{\tiny$\cev{e_3}$} (4.3,-.2) node{\tiny$\cev{e_2}$};

\draw (8.45,1.15) node {\tiny{$\vec{e_1}$}} (8.45,.3) node{\tiny$\cev{e}_1$} (7.5,-.52) node{\tiny$\vec{e}_3$} (7.7,-1.15) node{\tiny$\cev{e_3}$} (9.9,-.8) node{\tiny$\vec{e_2}$} (9.4,-1.3) node{\tiny$\cev{e}_2$};

\end{tikzpicture}
\end{center}
\caption{Example of fat graph and associated surface.}
\label{fat_graph_example_two_vertices}
\end{figure}

\end{example}

\begin{lemma}\label{lemma:2.6}
The number of boundary components of a fat graph $\Gamma=(E, \sim, \sigma_1, \sigma_0)$ is the same as the number of disjoint cycles in $\sigma_0^{-1}\sigma_1$.
\end{lemma}

\begin{proof}
For proof see Lemma 2.3 of \cite{sanki2018filling}.
\end{proof}

\section{Classification of curves on fat graphs}

Decorated fat graphs play an important role in our construction of fillings on surfaces. Such a graph can be written as an edge-disjoint union of its standard cycles. The set of standard cycles corresponds to a set of simple closed curves that gives a  filling. In this section, we classify the standard cycles of a decorated fat graph in Theorem~\ref{classif_of_std_cycle_in_fat_graph}. Before going to the theorem directly, let us consider the following motivating example.
\begin{example}\label{example}
Consider the filling pair $(\alpha, \beta)$ on $S_2$ as shown in Figure~\ref{sep_fill_on_S_2} and the corresponding fat graph $\Gamma(\alpha,\beta)=(E,\sim,\sigma_1,\sigma_0)$ as described below (see Figure~\ref{fatgraph_for_fill_on_S_2}): 
\begin{enumerate}
    \item $E=\{x_i,x_i^{-1},y_i,y_i^{-1}\mid i=1,\dots,6\}$.
    \item The equivalence classes of $\sim$ are 
    \begin{align*}
        v_1&=\{x_1,y_6^{-1},x_6^{-1},y_1\}, \ \ \ \ \ v_4=\{x_4,y_4,x_3^{-1},y_3^{-1}\}\\
        v_2&=\{x_2,y_2^{-1},x_1^{-1},y_3\}, \ \ \ \ \  v_5=\{x_5,y_2,x_4^{-1},y_1^{-1}\}\\
        v_3&=\{x_3,y_6,x_2^{-1},y_5^{-1}\}, \ \ \ \ \  v_6=\{x_6,y_4^{-1},x_5^{-1},y_5\}.
    \end{align*}
    \item $\sigma_1(x_i)=x_i^{-1}, \sigma_1(y_i)=y_i^{-1}, i=1,\dots,6.$
    \item $\sigma_0=\prod\limits_{i=1}^6\sigma_{v_i},$ where
    \begin{align*}
       \sigma_{v_1}&=(x_1,y_6^{-1},x_6^{-1},y_1), \ \ \ \ \ \sigma_{ v_4}=(x_4,y_4,x_3^{-1},y_3^{-1})\\
        \sigma_{v_2}&=(x_2,y_2^{-1},x_1^{-1},y_3), \ \ \ \ \ \sigma_{ v_5}=(x_5,y_2,x_4^{-1},y_1^{-1})\\
        \sigma_{v_3}&=(x_3,y_6,x_2^{-1},y_5^{-1}), \ \ \ \ \ \sigma_{v_6}=(x_6,y_4^{-1},x_5^{-1},y_5).
    \end{align*}
\end{enumerate}
In this graph, $\alpha=(x_1, \dots,x_6)$ and $\beta=(y_1,\dots,y_6)$ are the standard cycles.
The curve $\alpha$ is separating.

We construct a $6\times4$ matrix $M(\alpha,\beta)$ whose rows correspond to the cycles of the fat graph structure which are arranged such that edges $x_i$'s occur in the first position. For instance, the first row of $M(\alpha,\beta)$ is  $(x_1,y_6^{-1},x_6^{-1},y_1)$. The matrix is shown in Figure~\ref{fatgraph_matrix}. We call this matrix $M(\alpha,\beta)$ as the \emph{normal matrix} associated to the filling pair $(\alpha,\beta)$.
We observe that for each $i$, $y_i$ and $y_i^{-1}$ are in the same column.
\end{example}
For any fat graph with two standard cycles, its normal matrix is similarly defined. Furthermore, it has all the properties as discussed in Example~\ref{example}, when the two standard cycles are associated with a separating filling pair. Now, we are ready to state and prove the classification theorem.

\begin{figure}[htbp]
\begin{center}
\begin{tikzpicture}[xscale=1.5,yscale=1.5]

    \draw [blue] (0,0) ellipse (3cm and 1.2cm);
    \draw (-2,.05) to [bend right] (-1.2,.05);
    \draw (-2,.05) to [bend left] (-1.2,.05);
    
    \draw (1.2,.05) to [bend right] (2,.05);
    \draw (1.2,.05) to [bend left] (2,.05);
    
    \draw [red, -<-=.4,thick] (0,1.2) arc
        [
            start angle=90,
            end angle=270,
            x radius=.3cm,
            y radius =1.2cm
        ] ;
    \draw [dashed, red,thick] (0,-1.2) arc
        [
            start angle=270,
            end angle=450,
            x radius=.3cm,
            y radius =1.2cm
        ] ;
    \draw[red] (-.5,0) node {$\alpha$};
    \draw[green] (-1.3,-.8) node {$\beta$};
    \draw (1.1,.82) node {\footnotesize$y_1$};
    \draw (-1,-.32) node {\footnotesize$y_2$};
    \draw (2.8,0) node {\footnotesize$y_3$};
    \draw (-2,-.41) node {\footnotesize$y_4$};
    \draw (1.3,.27) node {\footnotesize$y_5$};
    \draw (-1,.88) node {\footnotesize$y_6$};
    
    \draw (-.3,1.1) node {\footnotesize$x_1$};
    \draw (.45,0) node {\footnotesize$x_2$};
    \draw (-.37,-.9) node {\footnotesize$x_3$};
    \draw (-.45,-.5) node {\footnotesize$x_4$};
    \draw (-.15,0) node {\footnotesize$x_5$};
    \draw (-.1,.67) node {\footnotesize$x_6$};
    
    \draw [cyan,-<-=.35,thick] (2.3,0) arc
        [
            start angle=0,
            end angle=180,
            x radius=2.5cm,
            y radius =.8cm
        ] ;
        
        \draw [cyan,thick] (-2.7,0) arc
        [
            start angle=180,
            end angle=258,
            x radius=2cm,
            y radius =1.12cm
        ] ;
    
    \draw [dashed, cyan,thick] (-1.2,-1.09) arc
        [
            start angle=265,
            end angle=360,
            x radius=2.2cm,
            y radius =1.12cm
        ] ;
    \draw [cyan,thick] (1.2,0) arc
        [
            start angle=0,
            end angle=260,
            x radius=1.7cm,
            y radius =.6cm
        ] ;
    \draw [cyan,thick] (-.8,-.59) arc
        [
            start angle=240,
            end angle=300,
            x radius=2.3cm,
            y radius =.5cm
        ] ;
    \draw [cyan,thick] (1.5,-.59) arc
        [
            start angle=300,
            end angle=360,
            x radius=2.3cm,
            y radius =.5cm
        ] ;
    \draw [cyan,thick] (2.65,-.18) arc
        [
            start angle=0,
            end angle=60,
            x radius=2.4cm,
            y radius =1.41cm
        ] ;
        
    \draw [dashed, cyan,thick] (1.45,1.045) arc
        [
            start angle=70,
            end angle=180,
            x radius=1.98cm,
            y radius =1.12cm
        ] ;
         
    \draw [cyan, thick] (-1.211,0) arc
        [
            start angle=180,
            end angle=360,
            x radius=1.752cm,
            y radius =.4cm
        ] ;
   
\end{tikzpicture}
\end{center}
\caption{Filling pair $(\alpha, \beta)$ on $S_2$} 
\label{sep_fill_on_S_2}
\end{figure}

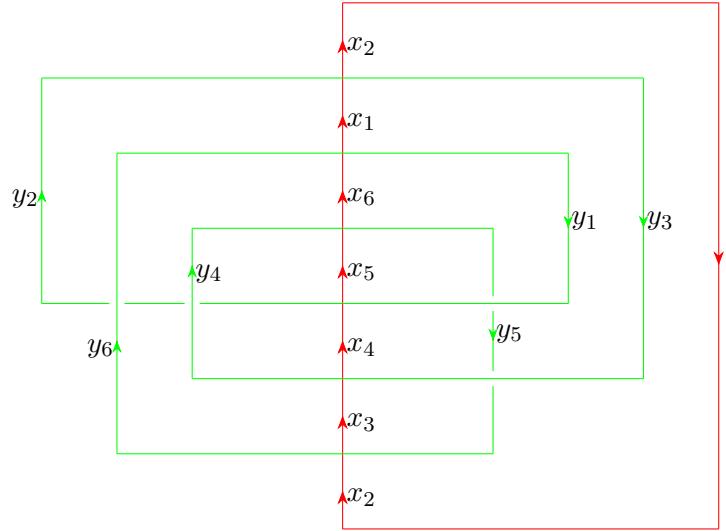
\begin{figure}[htbp]
\begin{center}
\begin{tikzpicture}[xscale=1,yscale=1]
    \draw[red] (0,0) to (0,7);
    \draw [green] (0,1) to (2,1);
    \draw [green] (0,2) to (4,2);
    \draw [green] (0,3) to (3,3);
    \draw [green] (0,4) to (2,4);
    \draw [green] (0,5) to (3,5);
    \draw [green] (0,6) to (4,6);
    \draw [green] (4,6) to (4,2);
    \draw [green] (3,3) to (3,5);
    \draw [green] (2,4) to (2,3.1);
    \draw [green] (2,2.9) to (2,2.1);
    \draw [green] (2,1.9) to (2,1);
    \draw[red] (0,0) to (5,0);
    \draw[red] (5,0) to (5,7);
    \draw[red] (5,7) to (0,7);
    \draw [green] (0,2) to (-2,2);
    \draw [green] (-2,2) to (-2,4);
    \draw [green] (-2,4) to (0,4);
    \draw [green] (0,1) to (-3,1);
    \draw [green] (-3,1) to (-3,5);
    \draw [green] (-3,5) to (0,5);
    \draw [green] (0,6) to (-4,6);
    \draw [green] (-4,6) to (-4,3);
    \draw [green] (-4,3) to (-3.1,3);
    \draw [green] (-2.9,3) to (-2.1,3);
    \draw [green] (-1.9,3) to (0,3);
    
    \draw[red] [-{Stealth[scale=1.1]}] (0,0.5) -- (0,0.51);
    \draw[red] [-{Stealth[scale=1.1]}] (0,1.5) -- (0,1.51);
    \draw[red] [-{Stealth[scale=1.1]}] (0,2.5) -- (0,2.51);
    \draw[red] [-{Stealth[scale=1.1]}] (0,3.5) -- (0,3.51);
    \draw [red][-{Stealth[scale=1.1]}] (0,4.5) -- (0,4.51);
    \draw [red][-{Stealth[scale=1.1]}] (0,5.5) -- (0,5.51);
    \draw [red][-{Stealth[scale=1.1]}] (0,6.5) -- (0,6.51);
    \draw[red] [-{Stealth[scale=1.1]}] (5,3.52) -- (5,3.51);
    
    \draw [red][-{Stealth[color=green]}] (4,4.01)--(4,4); 
    \draw[red] [-{Stealth[color=green]}] (3,4.01)--(3,4);
    \draw[red] [-{Stealth[color=green]}] (2,2.51)--(2,2.5);
    \draw [red][-{Stealth[color=green]}] (-3,2.5)--(-3,2.51);
    \draw [red][-{Stealth[color=green]}] (-4,4.5)--(-4,4.51);
    \draw[red] [-{Stealth[color=green]}] (-2,3.5)--(-2,3.51);
    
    \draw (.25,6.42) node {$x_2$};
    \draw (.25,5.42) node {$x_1$};
    \draw (.25,4.42) node {$x_6$};
    \draw (.25,3.42) node {$x_5$};
    \draw (.25,2.42) node {$x_4$};
    \draw (.25,1.42) node {$x_3$};
    \draw (.25,0.42) node {$x_2$};
    
    \draw (4.22,4.1) node{$y_3$};
    \draw (3.22,4.1) node{$y_1$};
    \draw (2.22,2.6) node{$y_5$};
    \draw (-3.22,2.4) node{$y_6$};
    \draw (-4.22,4.4) node{$y_2$};
    \draw (-1.78,3.4) node{$y_4$};
\end{tikzpicture}
\end{center}
\caption{Fat graph corresponding to the filling pair $(\alpha, \beta)$.}
\label{fatgraph_for_fill_on_S_2}
\end{figure}

\begin{figure}[htbp]
\begin{center}
\begin{tikzpicture}[xscale=1.5,yscale=1.5]
    \draw (2,-1) node {$y_3$} (-1,-1) node {$x_2$}  (0,-1) node {$y_2^{-1}$} (1,-1) node  {$x_1^{-1}$};
    \draw (2,-.5) node {$y_1$} (-1,-.5) node {$x_1$}  (0,-.5) node {$y_6^{-1}$} (1,-.5) node  {$x_6^{-1}$};
    \draw (2,-3) node {$y_5$} (-1,-3) node {$x_6$}  (0,-3) node {$y_4^{-1}$} (1,-3) node  {$x_5^{-1}$};
    \draw (2,-2.5) node {$y_1^{-1}$} (-1,-2.5) node {$x_5$}  (0,-2.5) node {$y_2$} (1,-2.5) node  {$x_4^{-1}$};
    \draw (2,-2) node {$y_3^{-1}$} (-1,-2) node {$x_4$}  (0,-2) node {$y_4$} (1,-2) node  {$x_3^{-1}$};
    \draw (2,-1.5) node {$y_5^{-1}$} (-1,-1.5) node {$x_3$}  (0,-1.5) node {$y_6$} (1,-1.5) node  {$x_2^{-1}$};
\end{tikzpicture}
\end{center}
\caption{The normal matrix $M(\alpha,\beta)$.} 
\label{fatgraph_matrix}
\end{figure}
\begin{theorem}\label{classif_of_std_cycle_in_fat_graph}
    Let $\Gamma=(E,\sim,\sigma_1,\sigma_0)$ be a fat graph with two standard cycles $\alpha=(x_1, \dots,x_n)$ and $\beta=(y_1,\dots,y_n)$. The cycle $\alpha$ is separating if and only if  $y_i$ and $y_i^{-1}$ are in the same column of the normal matrix $M(\alpha,\beta)$, for every $i=1,\dots,n$.
\end{theorem}

\begin{proof}
    Consider the surface $\Sigma(\Gamma)$, associated to the fat graph $\Gamma$. We construct a closed and oriented surface $S$  from $\Sigma(\Gamma)$ by attaching a disk along each boundary component. The simple closed curves on $S$ corresponding to the standard cycles $\alpha$ and $\beta$ are again denoted by $\alpha$ and $\beta$ respectively. If we cut the surface along $(\alpha \cup \beta)$, then we get a finite number of polygons with labelled boundary corresponding to the cycles of the permutation $\sigma_0^{-1}\sigma_1$. The curve $\alpha$ is separating if and only if whenever the polygons are identified along the edges $y_i$ and $y_i^{-1}$, then the resulting surface is still disconnected. 
    
    ($\Leftarrow$) Assume that the normal matrix $M(\alpha,\beta)$ is as in Theorem~\ref{classif_of_std_cycle_in_fat_graph}. We show that $\alpha$ is separating. Consider the disjoint cycle representation of the permutation $\sigma_0^{-1}\sigma_1$. We claim that the $y$-edges contained in each cycle are either from the second or fourth column of $M(\alpha,\beta)$,  but not from the both. The entries in each cycle are alternatively coming from $x$-edges and $y$-edges. Consider a generic cycle $(y_i,x_j,y_t,\dots)$ of the permutation $\sigma_0^{-1}\sigma_1$. To prove the claim, it suffices to show that $y_i$ and $y_t$ are in the same column. Suppose $y_i$ is in the second column, then $\sigma_1(y_i)=y_i^{-1}$ is also in the second column, by assumption on matrix $M(\alpha,\beta)$. Therefore, $\sigma_0^{-1} \left(y_i^{-1}\right) =x_j$ is in first column and $\sigma_1(x_j)=x_j^{-1}$ in third column, by the properties of $M(\alpha,\beta)$. Now, it follows that $\sigma_0^{-1}(x_j^{-1})$ is in the second column (see Figure~\ref{general_fat_graph_matrix} and Figure~\ref{multiplication_rule}). In case $y_i$ is in fourth column, a similar argument shows that $y_t$ is also in fourth column.
    
    Suppose $D_2$ and $D_4$ be the collection of the polygons whose y-edges are in second and fourth column respectively. Then, both $D_2$ and $D_4$ are nonempty disjoint sets. Therefore, after identifying the polygons along the y-edges, we get at least two disconnected sets. Hence, $\alpha$ is a separating curve. 
    
    ($\Rightarrow$) The above process is reversible and this completes the proof.
    \begin{figure}[htbp]\label{x}
    \begin{center}
    \begin{tikzpicture}[xscale=1.5,yscale=1.5]
        \draw (-2,0) node {C1} (-1,0) node {C2}  (0,0) node {C3} (1,0) node  {C4};
        \draw (-2,0) node {C1} (-1,0) node {C2}  (0,0) node {C3} (1,0) node  {C4};
        \draw (-2,-1) node {$x_1$} (-1,-1) node {-}  (0,-1) node {$x_{n}^{-1}$} (1,-1) node  {-};
        \draw (-2,-1.5) node {$x_2$} (-1,-1.5) node {-}  (0,-1.5) node {$x_{1}^{-1}$} (1,-1.5) node  {-};
        \draw (-1.5,-1.8) node {$\vdots$} (0,-1.8) node {$\vdots$};
        \draw (-2,-2.3) node {$x_r$} (-1,-2.3) node {$y_i$}  (0,-2.3) node {$x_{r-1}^{-1}$} (1,-2.3) node  {$y_s^{-1}$};
        \draw (-1.5,-2.6) node {$\vdots$} (0,-2.6) node {$\vdots$};
        \draw (-2,-3.1) node {$x_j$} (-1,-3.1) node {$y_i^{-1}$}  (0,-3.1) node {$x_{j-1}^{-1}$} (1,-3.1) node  {$y_k$};
        \draw (-2,-3.6) node {$x_{j+1}$} (-1,-3.6) node {$y_t$}  (0,-3.6) node {$x_{j}^{-1}$} (1,-3.6) node  {$y_u$};
        \draw (-1.5,-3.9) node {$\vdots$} (0,-3.9) node {$\vdots$};
        \draw (-2,-4.4) node {$x_n$} (-1,-4.4) node {-}  (0,-4.4) node {$x_{n-1}^{-1}$} (1,-4.4) node  {-};
    \end{tikzpicture}
    \end{center}
    \caption{Normal matrix} 
    \label{general_fat_graph_matrix}
    \end{figure}
    \begin{figure}[htbp]
    \begin{center}
    \begin{tikzpicture}[xscale=1,yscale=1]
        \draw (-8,0) node {$y_i$};
        \draw[dashed,red] (-8,0) circle (.3cm);
        \draw [-{Stealth[color=black]}] (-7.6,0)--(-6,0);
        \draw (-6.8,.2) node {$\sigma_1$};
        \draw (-5.6,0) node {$y_i^{-1}$};
        \draw [-{Stealth[color=black]}] (-5.2,0)--(-3.6,0);
        \draw (-4.4,.3) node {$\sigma_0^{-1}$};
        \draw (-3.2,0) node {$x_j$};
        \draw[dashed,red] (-3.2,0) circle(.3cm);
        \draw [-{Stealth[color=black]}] (-2.8,0)--(-1.2,0);
        \draw (-2,.2) node {$\sigma_1$};
        \draw (-.8,0) node {$x_j^{-1}$};
        \draw [-{Stealth[color=black]}] (-.4,0)--(1.2,0);
        \draw (.4,.3) node {$\sigma_0^{-1}$};
        \draw (1.6,0) node {$y_t$} (2.3,0) node { $\dots$};
        \draw[dashed,red] (1.6,0) circle(.3cm);
        \draw [-{Stealth[color=black]}] (-8,-.3)--(-8,-1.3);
        \draw (-8,-1.6) node {2nd column};
        \draw [-{Stealth[color=black]}] (-5.6,-.3)--(-5.6,-1.3);
        \draw (-5.6,-1.6) node {2nd column};
        \draw [-{Stealth[color=black]}] (-3.2,-.3)--(-3.2,-1.3);
        \draw (-3.2,-1.6) node {1st column};
        \draw [-{Stealth[color=black]}] (-.8,-.3)--(-.8,-1.3);
        \draw (-.8,-1.6) node {3rd column};
        \draw [-{Stealth[color=black]}] (1.6,-.3)--(1.6,-1.3);
        \draw (1.6,-1.6) node {2nd column};
    \end{tikzpicture}
    \end{center}
    \caption{Encircled elements lie on same cycle of $\sigma_0^{-1}\sigma_1$.} 
    \label{multiplication_rule}
    \end{figure}
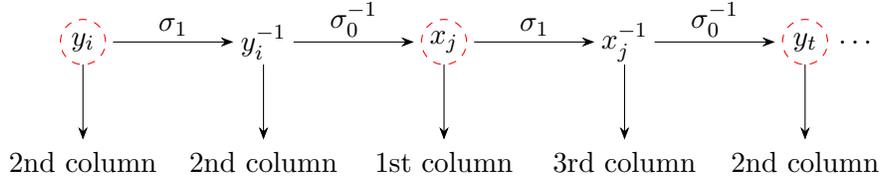
\end{proof}

\section{Existence of minimal separating filling pairs}
The goal of this section is to prove Theorem~\ref{theorem:2} which gives a necessary and sufficient condition for the existence of a minimally intersecting separating filling pair on $S_g$. We prove Proposition~\ref{prop:4.1} below which has a consequence giving the proof of the only if part of Theorem~\ref{theorem:2}.
\begin{proposition}\label{prop:4.1}
    Let $(\alpha, \beta)$ be a separating filling pair on $S_g$ and  $S_g\setminus \alpha=S_{g_1,1}\sqcup S_{g_2,1}$, where $g_1$ and $g_2$ are the genera of $S_{g_1,1}$ and $S_{g_2,1}$ respectively. Then the number of connected components in $S_{g_1,1}\setminus \beta$ and $S_{g_2,1} \setminus \beta$ are equal if and only if $g_1=g_2$.
\end{proposition}
\begin{proof}
    ($\Rightarrow$) Assume that $S_{g_1,1}\setminus\beta$ and $S_{g_2,1}\setminus\beta$ both have the same number of connected components, say $n$. The curve $\alpha$ being separating implies that the geometric intersection number $i(\alpha,\beta)$ is an even integer, say $2v$.  Now, the union $(\alpha\cup\beta)$ gives a cell decomposition of $S_{g_1,1}$ and $S_{g_2,1}$. In each cell decomposition, the number of $0$-cells is $2v$, the number of $1$-cells is $3v$ and the number of 2-cells is $n$. Euler's characteristic formula implies that
    \begin{align*}
    2v-3v+n&=2-2g_1-1\hspace{.5cm} \\
    \implies g_1&=(1-n+v)/2.
    \end{align*}
    A similar calculation gives $g_2=(1-n+v)/2.$
    Thus, we conclude $g_1=g_2$.
    
    ($\Leftarrow$) Suppose $g_1=g_2$. Applying Euler's characteristic formula, we have $g_1=(1-n_1+v)/2$ and $g_2=(1-n_2+v)/2$, where $n_i$'s are the number of components in $S_{g_i,1}\setminus \beta$, $i=1,2$ and $v=i(\alpha,\beta)/2$. Now, $g_1=g_2$ implies that $n_1=n_2$. This completes the proof.
\end{proof}
Next, we prove Lemma~\ref{lemma:4.2} which shows the non-existence of minimally intersecting separating filling pair on $S_2$ and as a consequence, we have $g\geq4$ is essential in Theorem~\ref{theorem:2}.  This also follows from the work of Jeffreys \cite{jeffreys2022meanders}, where he has shown that a filling pair on $S_2$ with a single separating curve requires at least $6$ intersections.

\begin{lemma}\label{lemma:4.2}
    There exists no minimally intersecting separating filling pair on $S_2$.
\end{lemma}
\begin{proof}
    The proof is by contradiction. Assume that there is a minimally intersecting separating filling pair $(\alpha,\beta)$ on $S_2$. Euler's characteristic formula implies that the geometric intersection number $i(\alpha,\beta)$ between the curves $\alpha$ and $\beta$ is $4$.  The curves $\alpha$ and $\beta$ correspond to two standard cycles, each of length $4$, in the decorated fat graph $\Gamma(\alpha,\beta)$. We denote them by $\alpha=(x_1,x_2,x_3,x_4)$ and $\beta=(y_1,y_2,y_3,y_4)$. Now, Theorem~\ref{classif_of_std_cycle_in_fat_graph} implies that the edges $y_i$ and $y_i^{-1}$ are in the same column of the normal matrix $M(\alpha,\beta)$. There is a unique such normal matrix up to relabelling (see Figure~\ref{possible_normal_matrix}) and in the corresponding fat graph $\Gamma(\alpha,\beta)$, there is a boundary component of length two which implies that the curves $\alpha$ and $\beta$ form a bigon on the surface $S_2$. This contradicts that the curves $\alpha$ and $\beta$ are in minimal position.
    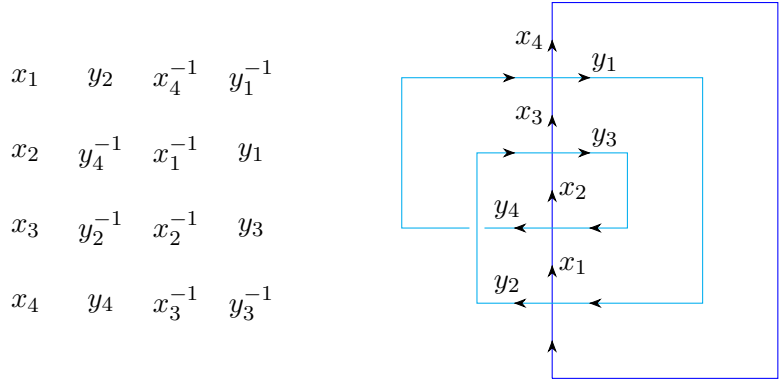
\begin{figure}[htbp]
    \begin{center}
    \begin{tikzpicture}[xscale=1,yscale=1]
        \draw (2,2) node {$y_1^{-1}$} (-1,2) node {$x_1$}  (0,2) node {$y_2$} (1,2) node  {$x_4^{-1}$};
        \draw (2,1) node {$y_1$} (-1,1) node {$x_2$}  (0,1) node {$y_4^{-1}$} (1,1) node  {$x_1^{-1}$};
        \draw (2,0) node {$y_3$} (-1,0) node {$x_3$}  (0,0) node {$y_2^{-1}$} (1,0) node  {$x_2^{-1}$};
        \draw (2,-1) node {$y_3^{-1}$} (-1,-1) node {$x_4$}  (0,-1) node {$y_4$} (1,-1) node  {$x_3^{-1}$};
        
        \draw[blue] (6,-2)--(9,-2)--(9,3)--(6,3)--(6,-2);
        \draw[cyan] (6,-1)--(5,-1)--(5,1)--(7,1)--(7,0)--(5.1,0);
        \draw[cyan] (4.9,0)--(4,0)--(4,2)--(8,2)--(8,-1)--(6,-1);
        \draw [-{Stealth[color=black]}] (6,-.5)--(6,-.49); 
        \draw [-{Stealth[color=black]}] (6,-1.5)--(6,-1.49); 
        \draw [-{Stealth[color=black]}] (6,.5)--(6,.51); 
        \draw [-{Stealth[color=black]}] (6,1.5)--(6,1.51); 
        \draw [-{Stealth[color=black]}] (6,2.5)--(6,2.51); 
        \draw [-{Stealth[color=black]}] (5.5,-1)--(5.49,-1); 
        \draw [-{Stealth[color=black]}] (5.5,0)--(5.49,0);
        \draw [-{Stealth[color=black]}] (5.5,1)--(5.51,1);
        \draw [-{Stealth[color=black]}] (5.5,2)--(5.51,2);
        \draw [-{Stealth[color=black]}] (6.5,-1)--(6.49,-1); 
        \draw [-{Stealth[color=black]}] (6.5,0)--(6.49,0);
        \draw [-{Stealth[color=black]}] (6.5,1)--(6.51,1);
        \draw [-{Stealth[color=black]}] (6.5,2)--(6.51,2);
        \draw (6.28,-.5) node{$x_1$};
        \draw (6.28,.5) node{$x_2$};
        \draw (5.7,1.5) node{$x_3$};
        \draw (5.7,2.5) node{$x_4$};
        \draw (5.4,-.75) node{$y_2$};
        \draw (5.4,{1-.75}) node{$y_4$};
        \draw (6.7,{3-.8}) node{$y_1$};
        \draw (6.7,{2-.8}) node{$y_3$};
    \end{tikzpicture}
    \end{center}
    \caption{Normal matrix $M(\alpha,\beta)$ and the corresponding fat graph.} 
    \label{possible_normal_matrix}
    \end{figure}
\end{proof}
\begin{proof}[Proof of Theorem~\ref{theorem:2}]\label{proof_th_2}
    ($\Rightarrow$) Suppose $(\alpha,\beta)$ is a minimally intersecting separating filling pair on $S_g$. Consider $S_{g_1,1},S_{g_2,1},g_1$ and $g_2$, as in Proposition~\ref{prop:4.1}. Then the number of components in each of $S_{g_1,1}\setminus \beta$ and $S_{g_2,1} \setminus \beta$ is one. Therefore, by Proposition~\ref{prop:4.1}, we have $g_1=g_2$ and this implies $g=g_1+g_2$, an even integer. Finally, by Lemma~\ref{lemma:4.2}, we have $g\geq 4$.
    
    ($\Leftarrow$) Suppose $g\geq 4$ is an even integer. We show that there is a minimal separating filling pair on $S_g$. The proof is by induction on $g$. The induction process closely follows that of Aougab-Huang used in the proof of Theorem 1.1 (Section 2.2, \cite{aougab2015minimally}). Here, we are required to construct new separating filling pairs for the base case $(g = 4)$ and for the inductive step (a new genus two piece).
    
    For $g=4$, the existence of a minimally intersecting separating filling pair follows from Figure~\ref{min_fill_pair_on_S_4}. We note that the fat graph of the curves $\alpha,\beta$, in Figure~\ref{min_fill_pair_on_S_4} can be used to check that $(\alpha,\beta)$ is indeed a separating filling pair.
    
    Let $(\alpha_g,\beta_g)$ be a minimal separating filling pair on $S_g$ (obtained by induction) and we prove the result for $g+2$.
    
    For the inductive step, let us consider the simple closed curves $\alpha'$ and $\beta'$ on $S_2$ as in Figure~\ref{a_sep_fill_on_S_2}. Then $S_2\setminus(\alpha'\cup\beta')$ is a disjoint union of 4 disks as shown in Figure~\ref{the_disjoint_union_of_4_disks}.
    
    \begin{figure}[htbp]
    \begin{center}
    \begin{tikzpicture}[xscale=.8,yscale=.6]
        \draw [] (0,0) ellipse (8cm and 4cm);       
        \draw [dashed, thick, red] (0,-4) arc
            [
                start angle=270,
                end angle=450,
                x radius=.8cm,
                y radius =4cm
            ] ;
        \draw [thick, red,->-=.45] (0,4) arc
            [
                start angle=90,
                end angle=270,
                x radius=.8cm,
                y radius =4cm
            ] ;
        \draw (-2.8,.05) .. controls (-3.5,-.3) .. (-4.2,.05);
        \draw (-2.9,-.05) .. controls (-3.5,.25) .. (-4.05,-.05);
        \draw (-5.8,.05) .. controls (-6.5,-.3) .. (-7.2,.05);
        \draw (-5.95,-.05) .. controls (-6.5,.25) .. (-7.05,-.05);
        
        \draw (2.8,.05) .. controls (3.5,-.3) .. (4.2,.05);
        \draw (2.9,-.05) .. controls (3.5,.25) .. (4.05,-.05);
        \draw (5.8,.05) .. controls (6.5,-.3) .. (7.2,.05);
        \draw (5.95,-.05) .. controls (6.5,.25) .. (7.05,-.05);

        \draw[thick, cyan,->-=.5] (-1,3.9) .. controls (-2,4.1) and (-8,2) .. (-7.5,-.4);
        \draw[thick, cyan] (-7.5,-.4) .. controls (-7.4,-1) and (-5,-1.3) .. (0,0);
        \draw[thick,cyan] (0,0) .. controls (3.5,.9) and (7.5,2) .. (7.5,-.1);
        \draw[thick,cyan] (7.5,-.1) .. controls (7.5,-.4) and (5,-3.5) .. (3.3,-3.62);
        \draw[dashed,thick,cyan] (3.3,-3.62) .. controls (0,-4) and (-5,-1) .. (-5,0);
         \draw [dashed, thick, cyan] (0,2) arc
            [
                start angle=90,
                end angle=180,
                x radius=5cm,
                y radius =2cm
            ] ;
         \draw [dashed, thick, cyan] (2.9,0) arc
            [
                start angle=0,
                end angle=90,
                x radius=2.9cm,
                y radius =2cm
            ] ;
        
        \draw[thick, cyan] (-1,-3.945) .. controls (.5,-4) and (3,-1) .. (2.9,0);
         \draw [dashed, thick, cyan] (-5.8,0) arc
            [
                start angle=180,
                end angle=260,
                x radius=5.8cm,
                y radius =4cm
            ] ;
        \draw[thick, cyan] (-5.8,0) .. controls (-5.8,.3)and (-3.5,.4) .. (-3.4,.2);
        \draw[dashed, thick, cyan] (-3.4,.2) .. controls (0,0.4) and (5,-3) .. (6.5,2.2);
        \draw[thick, cyan] (6.5,2.2) .. controls (5,4) and (-3.2,3) .. (-2.9,.05);
        
        \draw [dashed, thick, cyan] (-2.9,0) arc
            [
                start angle=180,
                end angle=270,
                x radius=2.9cm,
                y radius =2.6cm
            ] ;  
        \draw [dashed, thick, cyan] (0,-2.6) arc
            [
                start angle=270,
                end angle=360,
                x radius=5.9cm,
                y radius =2.6cm
            ] ;
        \draw[thick, cyan](5.9,0) .. controls (5.8,.3) and (4.5,.51) .. (3.5,.2);
        \draw[dashed, thick, cyan] (3.5,.2) .. controls (3,.3) and (2,3) .. (-1,3.9);
        
        \draw[red] (-1.05,1.1) node {$\alpha$};
        \draw[cyan] (3,2.7) node {$\beta$};
        
        \draw (-4,2.7) node{\tiny$\beta_1$} (5,1.4) node{\tiny$\beta_2$} (-3.7,-1.8) node{\tiny$\beta_3$} (2.7,-1.5) node{\tiny$\beta_4$} (-5.5,-2) node{\tiny$\beta_5$} (4.3,-.9) node{\tiny$\beta_6$} (-3.,1) node{\tiny$\beta_7$} (2.7,-2.65) node{\tiny$\beta_8$};
        \draw (-.86,3) node{\tiny$\alpha_1$} (-.45,.3) node{\tiny$\alpha_2$} (-1.1,-1.5) node{\tiny$\alpha_3$} (.6,-3.75) node{\tiny$\alpha_4$} (.2,-2.9) node{\tiny$\alpha_5$} (1,-2) node{\tiny$\alpha_6$} (1.1,.9) node{\tiny$\alpha_7$} (1.,2.3) node{\tiny$\alpha_8$};
    \end{tikzpicture}
    \end{center}
    \caption{$(\alpha,\beta)$ is a minimal separating filling pair on $S_4$.} 
    \label{min_fill_pair_on_S_4}
    \end{figure}
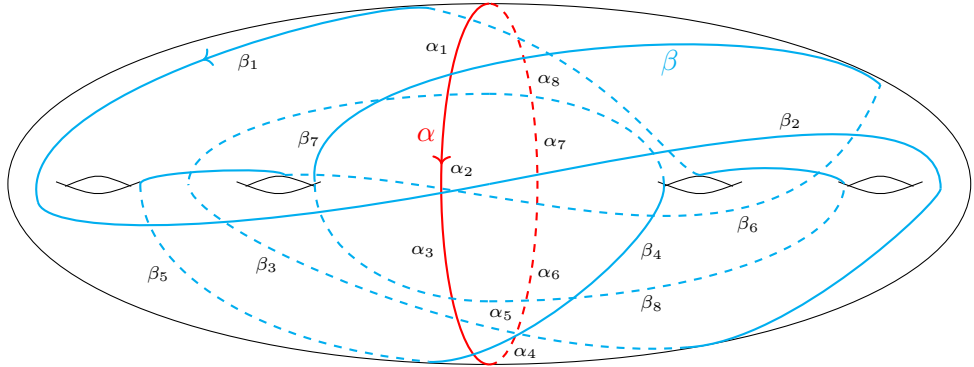
    
    \begin{figure}[htbp]
    \begin{center}
    \begin{tikzpicture}[xscale=2,yscale=1.7]
        \draw  (0,0) ellipse (3cm and 1.2cm);
        
        \draw (-2,.05) to [bend right] (-1.2,.05);
        \draw (-2,.05) to [bend left] (-1.2,.05);
        
        \draw (1.2,.05) to [bend right] (2,.05);
        \draw (1.2,.05) to [bend left] (2,.05);
        
        \draw [red,-<-=.5] (0,1.2) arc
            [
                start angle=90,
                end angle=270,
                x radius=.3cm,
                y radius =1.2cm
            ] ;
        \draw [dashed, red] (0,-1.2) arc
            [
                start angle=270,
                end angle=450,
                x radius=.3cm,
                y radius =1.2cm
            ] ;
        \draw (-.1,.9) node {\tiny $A$};
        \draw (-.15,.48) node {\tiny $B$};
        \draw (-.16,-.25) node {\tiny $C$};
        \draw (-.15,-.75) node {\tiny $D$};
        \draw (.3,-1) node {\tiny $E$};
        \draw (.3,.95) node {\tiny $F$};
        
        \draw[red] (-.4,.3) node {$\alpha'$};
        \draw[cyan,thick] (-1,.9) node {$\beta'$};
        
        \draw [-<-=.35, cyan,thick] (2.3,0) arc
            [
                start angle=0,
                end angle=150,
                x radius=2.5cm,
                y radius =.8cm
            ] ;
        \draw[cyan,thick] (-2.36,.402).. controls (-3.2,0)and(-2.2,-.95)..(-1.2,-1.09);
        
        \draw [dashed, cyan,thick] (-1.2,-1.09) arc
            [
                start angle=265,
                end angle=360,
                x radius=2.2cm,
                y radius =1.12cm
            ] ;
        \draw [cyan,thick] (1.2,0) arc
            [
                start angle=0,
                end angle=260,
                x radius=1.7cm,
                y radius =.6cm
            ] ;
        \draw [cyan,thick] (-.8,-.59) arc
            [
                start angle=240,
                end angle=300,
                x radius=2.3cm,
                y radius =.5cm
            ] ;
        \draw[cyan,thick] (1.5,-.59)..controls (3,-.5) and (3,.5)..(1.45,1.045);
            ] ;
        \draw [dashed, cyan,thick] (1.45,1.045) arc
            [
                start angle=70,
                end angle=180,
                x radius=1.98cm,
                y radius =1.12cm
            ] ;
        \draw [cyan,thick] (-1.211,0) arc
            [
                start angle=180,
                end angle=360,
                x radius=1.752cm,
                y radius =.4cm
            ] ;
        
        \draw (1.1,.8) node {\footnotesize$y_1$};
        \draw (-1,-.32) node {\footnotesize$y_2$};
        \draw (2.73,0) node {\footnotesize$y_3$};
        \draw (-2,-.41) node {\footnotesize$y_4$};
        \draw (1.3,.2) node {\footnotesize$y_5$};
        \draw (-1.6,.78) node {\footnotesize$y_6$};
        
        \draw (-.3,1.1) node {\footnotesize$x_1$};
        \draw (.45,0) node {\footnotesize$x_2$};
        \draw (-.37,-.9) node {\footnotesize$x_3$};
        \draw (-.45,-.5) node {\footnotesize$x_4$};
        \draw (-.15,0) node {\footnotesize$x_5$};
        \draw (-.38,.67) node {\footnotesize$x_6$};
    \end{tikzpicture}
    \end{center}
    \caption{Filling pair $(\alpha',\beta')$ on $S_2$.} 
    \label{a_sep_fill_on_S_2}
    \end{figure}
     
    \begin{figure}[htbp]
    \begin{center}
    \begin{tikzpicture}[xscale=.75,yscale=.75]
        \draw[red,->-=.5] (2,0) .. controls (1.5705,0.656) .. (1.4142,1.4142);
        \draw[thick,cyan,->-=.5] (1.4142,1.4142).. controls (0.6778,1.5705) .. (0,2);
        \draw[red,->-=.5] (0,2).. controls (-0.6778,1.5705) .. (-1.4142,1.4142);
        \draw[thick,cyan,-<-=.5] (-1.4142,1.4142).. controls (-1.5705,0.656) .. (-2,0);
        \draw[red,->-=.5] (-2,0).. controls (-1.5705,-0.656) .. (-1.4142,-1.4142);
        \draw[thick,cyan,-<-=.5] (-1.4142,-1.4142).. controls (-0.6778,-1.5705) .. (0,-2);
        \draw[red,->-=.5] (0,-2).. controls (0.6778,-1.5705) .. (1.4142,-1.4142);
        \draw[thick,cyan,->-=.5] (1.4142,-1.4142).. controls (1.5705,-0.656) .. (2,0);
        
        \draw (2.05,0.656) node {\tiny $x_2$} (0.7,1.9) node {\tiny $y_6$} (-0.7,1.9) node {\tiny $x_1$} (-2,0.6) node {\tiny $y_2$} (-1.95,-0.75) node {\tiny $x_5$} (-0.9,-1.85) node {\tiny $y_4$} (0.9,-1.85) node {\tiny $x_4$} (1.95,-0.75) node {\tiny $y_2$};
        
        \draw[thick, cyan,-<-=.5] (9,0) .. controls (8.5705,0.656) .. (8.4142,1.4142);
        \draw[red,-<-=.5] ({7+1.4142},1.4142).. controls ({7+0.6778},1.5705) .. ({7+0},2);
        \draw[thick, cyan,-<-=.5] ({7+0},2).. controls ({7-0.6778},1.5705) .. ({7-1.4142},1.4142);
        \draw[red,-<-=.5] ({7-1.4142},1.4142).. controls ({7-1.5705},0.656) .. ({7-2},0);
        \draw[thick, cyan,->-=.5] ({7-2},0).. controls ({7-1.5705},-0.656) .. ({7-1.4142},-1.4142);
        \draw[red,-<-=.5] ({7-1.4142},-1.4142).. controls ({7-0.6778},-1.5705) .. ({7+0},-2);
        \draw[thick, cyan,->-=.5] ({7+0},-2).. controls ({7+0.6778},-1.5705) .. ({7+1.4142},-1.4142);
        \draw[red,-<-=.5] ({7+1.4142},-1.4142).. controls ({7+1.5705},-0.656) .. ({7+2},0);
        
        \draw ({7+2.05},0.656) node {\tiny $y_5$} ({7+0.7},1.9) node {\tiny $x_5$} ({7-0.7},1.95) node {\tiny $y_1$} ({7-2},0.6) node {\tiny $x_6$} ({7-1.95},-0.75) node {\tiny $y_5$} ({7-0.9},-1.85) node {\tiny $x_2$} ({7+0.9},-1.85) node {\tiny $y_3$} ({7+1.95},-0.75) node {\tiny $x_3$};
        
        \draw[-<-=.5,cyan,thick] (2,-5) .. controls (0.8,-4.2) .. (0,-3);
        \draw[->-=.5,red] (0,-3).. controls (-.8,-4.2) .. (-2,-5);
        \draw[->-=.5,cyan,thick](-2,-5).. controls (-.8,-5.8) .. (0,-7);
        \draw[->-=.5,red](0,-7) .. controls (0.8,-5.8) .. (2,-5);
        
        \draw (1.3,-4) node {\tiny $y_6$} (-1.2,-4) node {\tiny $x_3$} (-1,-6.2) node {\tiny $y_4$} (1.2,-6.2) node {\tiny $x_6$};
        
        \draw[-<-=.5,red] (9,-5) .. controls (7.8,-4.2) .. (7,-3);
        \draw[-<-=.5,cyan,thick] (7,-3).. controls (6.2,-4.2) .. (5,-5);
        \draw[-<-=.5,red](5,-5).. controls (6.2,-5.8) .. (7,-7);
        \draw[->-=.5,cyan,thick](7,-7) .. controls (7.8,-5.8) .. (9,-5);
        
        \draw (8.3,-4) node {\tiny $x_4$} (5.8,-4) node {\tiny $y_3$} (6,-6.2) node {\tiny $x_1$} (8.2,-6.2) node {\tiny $y_1$};
        
        \draw (0,2.2) node {\tiny A} (1.57,1.57) node{\tiny E} (2.2,0) node{\tiny F}(1.57,-1.57) node{\tiny C} (-.1,-2.2) node{\tiny B} (-1.57,-1.57) node{\tiny D} (-2.2,0) node{\tiny C} (-1.57,1.57) node{\tiny F};
        \draw (7,2.2) node {\tiny C} (8.57,1.57) node{\tiny B} (9.2,0) node{\tiny E}(8.57,-1.57) node{\tiny D} (6.9,-2.2) node{\tiny F} (5.43,-1.57) node{\tiny E} (4.8,0) node{\tiny B} (5.43,1.57) node{\tiny A};
        \draw (.1,-2.8) node {\tiny E} (2.2,-5) node{\tiny A} (0,-7.2) node{\tiny B} (-2.2,-5) node {\tiny D};
        \draw (7.1,-2.8) node {\tiny D} (9.2,-5) node{\tiny C} (7,-7.2) node{\tiny A} (4.8,-5) node {\tiny F};
    \end{tikzpicture}
    \end{center}
    \caption{Topological disks after cutting $S_2$ along $\alpha'\cup\beta'$.} 
    \label{the_disjoint_union_of_4_disks}
    \end{figure}
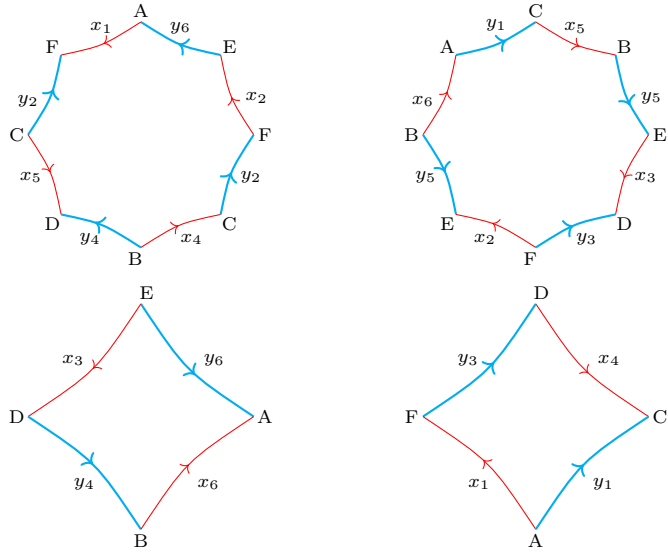
    The point $A$ lies on all of the four polygons. Let us extract an open ball around $A$ which contains no other intersection points. We extract another open ball around an intersection point of $\alpha_g$ and $\beta_g$ in a similar way. We attach these two surfaces along the boundary in such a way that the end points of the arc $\alpha_g$ are identified with the end points of the arc of $\alpha'$ and the same holds for the $\beta$ arcs. Then we get a filling pair $(\alpha_{g+2},\beta_{g+2})$ on $S_{g+2}$ with $\alpha_{g+2}$ separating and with two disks in the complement. Hence the proof.
\end{proof}
\subsection{Existence of minimal separating filling pairs for odd genera}
It follows from Proposition~\ref{prop:4.1} that there exists no separating filling pair on $S_g$ with two complementary disks, if $g=2n+1, n\in \N$. Furthermore, for any separating filling pair $(\alpha,\beta)$, the number of connected components in $S_g\setminus \alpha\cup\beta$ is even and at least four or equivalently $i(\alpha,\beta)\geq (2g+2)$. In this subsection, we prove the Theorem \ref{thm:new}.

\begin{proof}[Proof of Theorem~\ref{thm:new}]
    The proof is by induction on the genus $g$. For $g=3$, we construct a separating filling pair $(a_3,\beta_3)$ on $S_3$ satisfying the theorem. Consider the fat graph $\Gamma=(E,\sim,\sigma_1,\sigma_0)$ described below (see Figure~\ref{fig:4.5}):
\begin{enumerate}
    \item $E=\{a_i,a_i^{-1},b_i,b_i^{-1}|i=1,\dots,8\}$.
    \item The equivalence classes of $\sim$ are 
    \begin{align*}
        &v_1=\{a_1,b_1,a_8^{-1},b_8^{-1}\},\ \ \ \ \    v_5=\{a_5,b_7^{-1},a_4^{-1},b_8\}\\
        &v_2=\{a_2,b_3^{-1},a_1^{-1},b_4\}, \ \ \ \ \ v_6=\{a_6,b_1^{-1},a_5^{-1},b_2\}\\
        &v_3=\{a_3,b_7,a_2^{-1},b_6^{-1}\},\ \ \ \ \ v_7=\{a_7,b_5^{-1},a_6^{-1},b_6\}\\
        &v_4=\{a_4,b_5,a_3^{-1},b_4^{-1}\},\ \ \ \ \ v_8=\{a_8,b_3,a_7^{-1},b_2^{-1}\}.
    \end{align*}
    \item $\sigma_1(a_i)=a_i^{-1},\sigma_1(b_i)=b_i^{-1}$, for all $i=1,\dots,8.$
    \item $\sigma_0=\prod\limits_{i=1}^8\sigma_{v_i}$, where 
    \begin{align*}
        &\sigma_{v_1}=(a_1,b_1,a_8^{-1},b_8^{-1}),\ \ \ \ \  \sigma_{v_5}=(a_5,b_7^{-1},a_4^{-1},b_8)\\
        &\sigma_{v_2}=(a_2,b_3^{-1},a_1^{-1},b_4),\ \ \ \ \ \sigma_{v_6}=(a_6,b_1^{-1},a_5^{-1},b_2)\\
        &\sigma_{v_3}=(a_3,b_7,a_2^{-1},b_6^{-1}),\ \ \ \ \ \sigma_{v_7}=(a_7,b_5^{-1},a_6^{-1},b_6)\\
        &\sigma_{v_4}=(a_4,b_5,a_3^{-1},b_4^{-1}),\ \ \ \ \ 
        \sigma_{v_8}=(a_8,b_3,a_7^{-1},b_2^{-1}).
    \end{align*}
\end{enumerate}
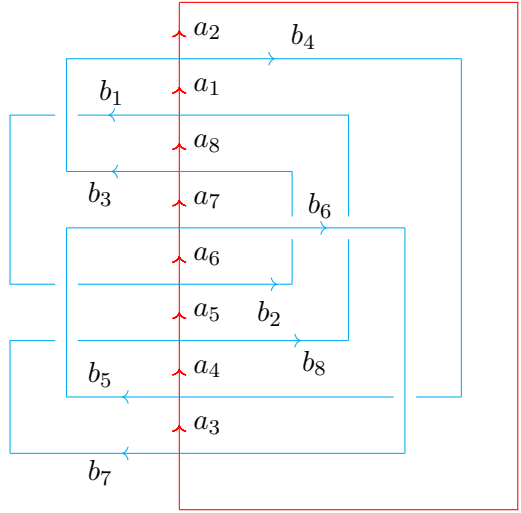
\begin{figure}[htbp]
    \centering
    \begin{tikzpicture}[xscale=1.5,yscale=1.5]
        \draw[red,->-={3/18},->-={5/18},->-={7/18},->-={9/18},->-={11/18},->-={13/18},->-={15/18},->-={17/18},] (0,0) to (0,4.5);
        \draw[red] (0,4.5)--(3,4.5)--(3,0)--(0,0);
        \draw[cyan] (0,.5)--(2,.5);
        \draw[cyan] (2,.5)--(2,2.5);
        \draw[cyan,-<-=.25] (2,2.5)--(-1,2.5);
        \draw[cyan] (-1,2.5)--(-1,1);
        \draw[cyan,-<-=.19](-1,1)--(1.9,1);
        \draw[cyan] (2.1,1)--(2.5,1);
        \draw[cyan] (2.5,1)--(2.5,4);
        \draw[cyan,-<-=.38] (2.5,4)--(-1,4)--(-1,3);
        \draw[cyan,-<-=.19] (-1,3)--(1,3)--(1,2.6);
        \draw[cyan] (1,2.4)--(1,2);
        \draw[cyan,-<-=.1] (1,2)--(-.9,2);
        \draw[cyan] (-1.1,2)--(-1.5,2)--(-1.5,3.5)--(-1.1,3.5);
        \draw[cyan,-<-=.1] (-.9,3.5)--(1.5,3.5)--(1.5,2.6);
        \draw[cyan] (1.5,2.4)--(1.5,1.5);
        \draw[cyan,-<-=.2] (1.5,1.5)--(-.9,1.5);
        \draw[cyan] (-1.1,1.5)--(-1.5,1.5)--(-1.5,.5);
        \draw[cyan,-<-=.7] (-1.5,.5)--(0,.5);

        \draw (.25,.75) node {$a_3$} (.25,1.25) node {$a_4$} (.25,1.75) node {$a_5$} (.25,2.25) node {$a_6$} (.25,2.75) node {$a_7$} (.25,3.25) node {$a_8$} (.25,3.75) node {$a_1$} (.25,4.25) node {$a_2$};
        \draw (-.6,3.7) node{$b_1$} (.8,1.74) node {$b_2$} (-.7,2.8) node{$b_3$} (1.1,4.2) node {$b_4$} (-.7,1.2) node {$b_5$} (1.25,2.7) node {$b_6$} (-.7,.33) node {$b_7$} (1.2,1.3) node {$b_8$};
    \end{tikzpicture}
    \caption{Fat graph for minimal separating filling pair $(\alpha_3,\beta_3)$ on $S_3$.}
    \label{fig:4.5}
\end{figure}
    It follows from Lemma~\ref{lemma:2.6} that the surface $\Sigma(\Gamma)$ has four boundary components as $\sigma_0^{-1}\sigma_1$ is a product of four disjoint cycles. An Euler characteristic argument implies that $\Sigma(\Gamma)$ has genus $3$. Therefore, after capping the boundary components by topological disks, we obtain a closed surface $S_3$ of genus $3$ on which the standard cycles $(a_1,\dots,a_8)$ and $(b_1,\dots,b_8)$ project to $2$ simple closed curves $\alpha_3$ and $\beta_3$. Applying Theorem \ref{classif_of_std_cycle_in_fat_graph}, we have $\alpha_3$ is separating and hence $(\alpha_3,\beta_3)$ is a required separating filling pair on $S_3$.
    
    Now, suppose $g\geq3$ odd and $(\alpha_g,\beta_g)$ is a separating filling pair on $S_g$ satisfying the theorem. Then $S_g\setminus(\alpha_g\cup\beta_g)$ is a disjoint union of four polygons. Choose a point of $\alpha_g\cap\beta_g$ which is a common vertex of three polygons. Consider the separating filling pair $(\alpha',\beta')$ on $S_2$, as described in Figure~\ref{a_sep_fill_on_S_2} and choose $A\in \alpha'\cap\beta'$. Then, we excise small disks around these two vertices and take connected sum in a similar way as discussed in the previous section to obtain a minimal separating filling pair $(\alpha_{g+2},\beta_{g+2})$ on $S_{g+2}$.
\end{proof}

\begin{section}{Counting of Mapping Class Group Orbits.}
The mapping class group $\M$ acts on the set $\mathcal{C}_g$ of all minimally intersecting separating filling pairs of $S_g$ as follows: $f\cdot (\alpha, \beta)=(f(\alpha), f(\beta))$, where $f\in \M$ and $(\alpha, \beta)\in \mathcal{C}_g$. In this section, we find an upper bound of the number of orbits of this action, in particular, we prove Theorem~\ref{thm:1.3}. The idea of the proof is similar to the proof of Theorem 1.1 \cite{aougab2015minimally}  by Aougab-Huang. More precisely, we use the method and language of filling permutations originally used by Aougab-Huang but we need to investigate the algebraic conditions characterising separating filling pairs.

\begin{subsection}{Oriented filling pair and filling permutation}\label{sec:5.1} Consider $(\alpha,\beta)\in \mathcal{C}_g$. By Euler characteristic formula, we have that the geometric intersection number between $\alpha$ and $\beta$ is $2g$.
These intersection points decompose $\alpha$ into $2g$ sub-arcs. We choose an initial arc and an orientation on $\alpha$. Now, we label the sub-arcs by $\alpha_1,\dots,\alpha_{2g},$ in accordance with the orientation and with $\alpha_1$ being the initial arc. Similarly, we label the sub-arcs of $\beta$ by $\beta_1, \dots,\beta_{2g}$. A filling pair $(\alpha,\beta)$ with a chosen orientation on each of $\alpha$ and $\beta$ is called an \textit{oriented filling pair} and an oriented filling pair with a labelling, we call as a \textit{labelled filling pair}. 

Let $\Sigma_g$ denote the group of permutations on $\mathrm{N}_{g}= \{1, 2, \dots, 8g\}$. We define a permutation $\phi^{(\alpha,\beta)}\in \Sigma_g$ corresponding to a labelled filling pair $(\alpha,\beta)$ as described below. Consider the ordered set
$$ A_g=\left\{\alpha_1,\beta_1,\hdots,\alpha_{2g},\beta_{2g},\alpha_1^{-1},\beta_1^{-1},\hdots,\alpha_{2g}^{-1},\beta_{2g}^{-1}\right\}.$$
If we cut the surface $S_g$ along $(\alpha\cup\beta)$, then we get two $4g$-gons whose sides are labelled by the members of $A_g$. Now, we define,
\begin{equation*}
    \phi^{(\alpha,\beta)}(i)=j;
\end{equation*}
if the $i^{\text{th}}$ element of $A_g$ is succeeded by the $j^{\text{th}}$ element of $A_g$ along the clockwise direction of the $4g$-gons.

\begin{example}
Consider the labelled filling pair $(\alpha,\beta)$ as in Figure~\ref{min_fill_pair_on_S_4}. The labelled polygons obtained by cutting the surface along $\alpha\cup\beta$ are shown in Figure~\ref{One_of_two_poly}. Then,
\begin{align*}
\phi^{(\alpha,\beta)}&=c_1c_2, \text{ where }\\  c_1&=(4,23,24,29,12,17,32,25,20,19,28,27,16,31,8,21)\text{ and}\\
c_2&=(1,14,11,26,7,6,15,2,5,10,13,22,9,30,3,18).
\end{align*}
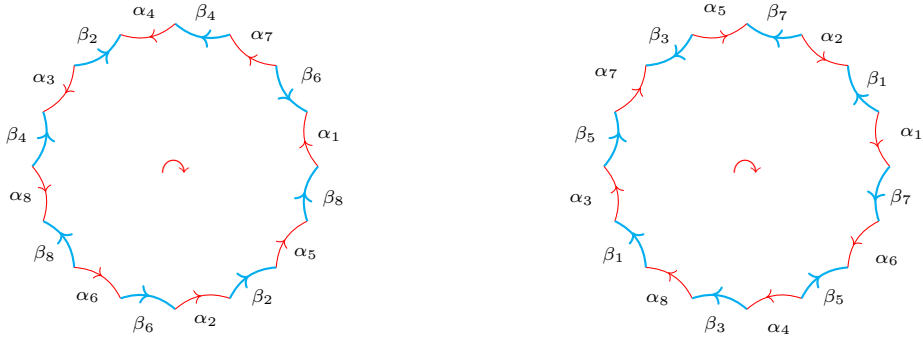
\begin{figure}[htbp]
\begin{center}
\begin{tikzpicture}[xscale=1.9,yscale=1.9]

    \draw[red,->-=.5] ({cos(0)},{sin(0)}) to [bend left] ({cos(22.5)},{sin(22.5)});
    \draw[thick, cyan, -<-=.5] ({cos(22.5)},{sin(22.5)}) to [bend left] ({cos(45)},{sin(45)});
    \draw[red, ->-=.5] ({cos(45)},{sin(45)}) to [bend left] ({cos(67.5)},{sin(67.5)});
    \draw[thick, cyan, ->-=.5] ({cos(67.5)},{sin(67.5)}) to [bend left] ({cos(90)},{sin(90)});
    \draw[red, ->-=.5] ({cos(90)},{sin(90)}) to [bend left] ({cos(112.5)},{sin(112.5)});
    \draw[thick, cyan, -<-=.5] ({cos(112.5)},{sin(112.5)}) to [bend left] ({cos(135)},{sin(135)});
    \draw[red, ->-=.5] ({cos(135)},{sin(135)}) to [bend left] ({cos(157.5)},{sin(157.5)});
    \draw[thick, cyan, -<-=.5] ({cos(157.5)},{sin(157.5)}) to [bend left] (-1,0);
    \draw[red, ->-=.5] (-1,0) to [bend left] ({cos(202.5)},{sin(202.5)});
    \draw[thick, cyan, -<-=.5] ({cos(202.5)},{sin(202.5)}) to [bend left] ({cos(225)},{sin(225)});
    \draw[red, ->-=.5] ({cos(225)},{sin(225)}) to [bend left] ({cos(247.5)},{sin(247.5)});
    \draw[thick, cyan, ->-=.5] ({cos(247.5)},{sin(247.5)}) to [bend left] ({cos(270)},{sin(270)});
    \draw[red, ->-=.5] ({cos(270)},{sin(270)}) to [bend left] ({cos(292.5)},{sin(292.5)});
    \draw[thick, cyan, ->-=.5] ({cos(292.5)},{sin(292.5)}) to [bend left] ({cos(315)},{sin(315)});
    \draw [red, ->-=.5]({cos(315)},{sin(315)}) to [bend left] ({cos(337.5)},{sin(337.5)});
    \draw [thick, cyan, ->-=.5]({cos(337.5)},{sin(337.5)}) to [bend left] ({cos(360)},{sin(360)});

    \draw[red, -<-=.5] ({4+cos(0)},{sin(0)}) to [bend left] ({4+cos(22.5)},{sin(22.5)});
    \draw[thick, cyan, ->-=.5] ({4+cos(22.5)},{sin(22.5)}) to [bend left] ({4+cos(45)},{sin(45)});
    \draw[red, -<-=.5] ({4+cos(45)},{sin(45)}) to [bend left] ({4+cos(67.5)},{sin(67.5)});
    \draw[thick, cyan, ->-=.5] ({4+cos(67.5)},{sin(67.5)}) to [bend left] ({4+cos(90)},{sin(90)});
    \draw[red, -<-=.5] ({4+cos(90)},{sin(90)}) to [bend left] ({4+cos(112.5)},{sin(112.5)});
    \draw[thick, cyan, ->-=.5] ({4+cos(112.5)},{sin(112.5)}) to [bend left] ({4+cos(135)},{sin(135)});
    \draw[red, -<-=.5] ({4+cos(135)},{sin(135)}) to [bend left] ({4+cos(157.5)},{sin(157.5)});
    \draw[thick, cyan, -<-=.5] ({4+cos(157.5)},{sin(157.5)}) to [bend left] (4-1,0);
    \draw[red, -<-=.5] (4+-1,0) to [bend left] ({4+cos(202.5)},{sin(202.5)});
    \draw[thick, cyan, -<-=.5] ({4+cos(202.5)},{sin(202.5)}) to [bend left] ({4+cos(225)},{sin(225)});
    \draw[red, -<-=.5] ({4+cos(225)},{sin(225)}) to [bend left] ({4+cos(247.5)},{sin(247.5)});
    \draw[thick, cyan, -<-=.5] ({4+cos(247.5)},{sin(247.5)}) to [bend left] ({4+cos(270)},{sin(270)});
    \draw[red, -<-=.5] ({4+cos(270)},{sin(270)}) to [bend left] ({4+cos(292.5)},{sin(292.5)});
    \draw[thick, cyan, ->-=.5] ({4+cos(292.5)},{sin(292.5)}) to [bend left] ({4+cos(315)},{sin(315)});
    \draw [red, -<-=.5]({4+cos(315)},{sin(315)}) to [bend left] ({4+cos(337.5)},{sin(337.5)});
    \draw [thick, cyan, -<-=.5]({4+cos(337.5)},{sin(337.5)}) to [bend left] ({4+cos(360)},{sin(360)});

    \draw ({1.1*cos(11.25)},{1.1*sin(11.25)}) node{\tiny$\alpha_1$} ({1.15*cos(33.75)},{1.15*sin(33.75)}) node {\tiny$\beta_6$} ({1.1*cos(56.25)},{1.1*sin(56.25)}) node{\tiny$\alpha_7$} ({1.1*cos(78.75)},{1.1*sin(78.75)}) node {\tiny$\beta_4$} ({1.1*cos(101.25)},{1.1*sin(101.25)}) node{\tiny$\alpha_4$} ({1.1*cos(123.75)},{1.1*sin(123.75)}) node {\tiny$\beta_2$}   ({1.1*cos(146.25)},{1.1*sin(146.25)}) node{\tiny$\alpha_3$} ({1.13*cos(168.75)},{1.13*sin(168.75)}) node {\tiny$\beta_4$}  ({1.1*cos(191.25)},{1.1*sin(191.25)}) node{\tiny$\alpha_8$} ({1.1*cos(213.75)},{1.1*sin(213.75)}) node {\tiny$\beta_8$}    ({1.1*cos(236.25)},{1.1*sin(236.25)}) node{\tiny$\alpha_6$} ({1.13*cos(258.75)},{1.13*sin(258.75)}) node {\tiny$\beta_6$}   ({1.1*cos(281.25)},{1.1*sin(281.25)}) node{\tiny$\alpha_2$} ({1.1*cos(303.75)},{1.1*sin(303.75)}) node {\tiny$\beta_2$}   ({1.1*cos(326.25)},{1.1*sin(326.25)}) node{\tiny$\alpha_5$} ({1.13*cos(348.75)},{1.13*sin(348.75)}) node {\tiny$\beta_8$};
    
    \draw ({4+1.18*cos(11.25)},{1.188*sin(11.25)}) node{\tiny$\alpha_1$} ({4+1.1*cos(33.75)},{1.1*sin(33.75)}) node {\tiny$\beta_1$} ({4+1.08*cos(56.25)},{1.08*sin(56.25)}) node{\tiny$\alpha_2$} ({4+1.1*cos(78.75)},{1.1*sin(78.75)}) node {\tiny$\beta_7$} ({4+1.12*cos(101.25)},{1.12*sin(101.25)}) node{\tiny$\alpha_5$} ({4+1.1*cos(123.75)},{1.1*sin(123.75)}) node {\tiny$\beta_3$}   ({4+1.18*cos(146.25)},{1.18*sin(146.25)}) node{\tiny$\alpha_7$} ({4+1.15*cos(168.75)},{1.15*sin(168.75)}) node {\tiny$\beta_5$}  ({4+1.18*cos(191.25)},{1.18*sin(191.25)}) node{\tiny$\alpha_3$} ({4+1.12*cos(213.75)},{1.12*sin(213.75)}) node {\tiny$\beta_1$}    ({4+1.12*cos(236.25)},{1.12*sin(236.25)}) node{\tiny$\alpha_8$} ({4+1.13*cos(258.75)},{1.13*sin(258.75)}) node {\tiny$\beta_3$}   ({4+1.16*cos(281.25)},{1.16*sin(281.25)}) node{\tiny$\alpha_4$} ({4+1.1*cos(303.75)},{1.1*sin(303.75)}) node {\tiny$\beta_5$}   ({4+1.18*cos(326.25)},{1.18*sin(326.25)}) node{\tiny$\alpha_6$} ({4+1.1*cos(348.75)},{1.1*sin(348.75)}) node {\tiny\tiny $\beta_7$};

    \draw[red] (0,0) node{$\curvearrowright$};
    \draw[red] (4,0) node{$\curvearrowright$};

\end{tikzpicture}
\end{center}
\caption{The labelled polygons.}
\label{One_of_two_poly}
\end{figure}

\end{example}
Now, we study the properties of the permutation $\phi^{(\alpha,\beta)}$ in the following proposition.

\begin{proposition}\label{prop_of_phi_alpha_beta}
The permutation $\phi^{(\alpha,\beta)}$ has the following properties.
\begin{enumerate}
     \item $\phi^{(\alpha,\beta)}$ is parity respecting and sends even entries to odds and vice-versa.
     \item $\phi^{(\alpha,\beta)}$ is a product of two disjoint $4g$-cycles.
  
    \item One cycle of $\phi^{(\alpha,\beta)}$ contains each element of $\{1,3,\hdots,4g-1\}$ and the other contains each element of $\{4g+1,4g+3,\hdots,8g-1\}$.
    \item If a cycle contains some even integer $2i$, then it contains $(2i+4g) (\mathrm{mod}\  8g).$
    \item If a cycle contains an even integer of the form $(4k+2)$, then all the even integers in this cycle have this form. A similar statement is true for the integers of the form $4k$.
    \item $\phi^{(\alpha,\beta)}$ satisfies the equation $\phi^{(\alpha,\beta)}Q^{4g}\phi^{(\alpha,\beta)}=\tau$,
    where 
    \begin{align*}
    Q &=(1,2,\hdots,8g) \text{ and}\\ \tau &=\tau_1\tau_2\tau_3\tau_4, \text{ where }\\ 
    \tau_1&= (1,3,\hdots,4g-1),\\
    \tau_2&=(2,4,\hdots,4g),\\
    \tau_3&=(8g-1,8g-3,\hdots,4g+1) \text{ and }\\
    \tau_4&=(8g,8g-2,\dots,4g+2).
    \end{align*}
\end{enumerate}
\end{proposition}
\begin{proof}
Let $P_1$ and $P_2$ be the $4g$-gons obtained by cutting the surface $S_g$ along $\alpha\cup \beta.$
\begin{enumerate}
    \item Among every two consecutive sides of the $4g$-gons $P_1$ and $P_2$, one comes from the $\alpha$-arcs and the other from the $\beta$-arcs, which implies that $\phi^{(\alpha,\beta)}$ is parity respecting and sends even entries to odds and vice-versa.
\item
Each polygon corresponds to a cycle of $\phi^{(\alpha,\beta)}$ and the converse is also true. Now, the statement follows from the minimality of the separating filling pair $(\alpha,\beta)$.
\item
By consideration, the curve $\alpha$ is separating which implies that
\begin{figure}[htbp]
\begin{center}
\begin{tikzpicture}[xscale=1,yscale=1]


\draw (0,0) ellipse (4cm and 1.2cm);

\draw [red] (0,1.2) arc
    [
        start angle=90,
        end angle=270,
        x radius=.3cm,
        y radius =1.2cm
    ] ;
    \draw [dashed, red] (0,-1.2) arc
    [
        start angle=270,
        end angle=450,
        x radius=.3cm,
        y radius =1.2cm
    ] ;
\draw (-3.5,.1) .. controls (-3,-.2) .. (-2.5,.1);
\draw (-3.35,0) .. controls (-3,.1) .. (-2.65,0);

\draw (-1.5,.1) .. controls (-1,-.2) .. (-.5,.1);
\draw (-1.35,0) .. controls (-1,.1) .. (-.65,0);


\draw (3.5,.1) .. controls (3,-.2) .. (2.5,.1);
\draw (3.35,0) .. controls (3,.1) .. (2.65,0);

\draw (1.5,.1) .. controls (1,-.2) .. (.5,.1);
\draw (1.35,0) .. controls (1,.1) .. (.65,0);


\draw [red] (-.1,-.2) node {$\alpha$};

\draw [thick, cyan] (-1.5,.5) .. controls (0,.7) .. (1.5,.5);
 
\draw [cyan] (-1,.8) node {$\beta_1$} (1,.8) node {$\beta_2$};

\end{tikzpicture}
\end{center}
\caption{Separating Curve $\alpha$ on $S_g$.} 
\label{sep_curve_alpha_on_S_g}
\end{figure}
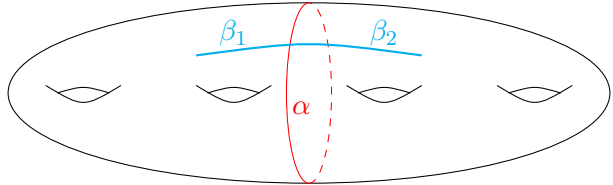
$\alpha_1,\hdots,\alpha_{2g}$ are in one polygon and $\alpha_1^{-1},\hdots,\alpha_{2g}^{-1}$ are in the other polygon.

\item The proof follows from the fact that for each $i$, the sides labelled by $\beta_i$ and $\beta_i^{-1}$ are in the same polygon (we refer to Theorem \ref{classif_of_std_cycle_in_fat_graph}). 
\item The edges labelled by $\beta_i$ and $\beta_i^{-1}$, for $i\in\{1,\dots,2g\}$ odd integers, are sides of one polygon and for even $i$'s, they are the sides of the other polygon (see Figure~\ref{sep_curve_alpha_on_S_g}). Therefore, the statement follows from Theorem \ref{classif_of_std_cycle_in_fat_graph}.
\item The proof follows from Theorem 2.1~\cite{nieland2018connected}.
\end{enumerate}
\end{proof}

If a permutation $\phi\in\Sigma_{8g}$ satisfies the properties $(1)-(6)$ of Proposition \ref{prop_of_phi_alpha_beta}, then we call it a \textit{ filling permutation}. Given a minimally intersecting separating filling pair $(\alpha,\beta)$ on $S_g$, we have  that  $\phi=\phi^{(\alpha,\beta)}$ is a filling permutation. Conversely, in Theorem \ref{thm:5.3} below, we prove that every filling permutation is realized by a minimally intersecting separating filling pair.

\begin{theorem} \label{thm:5.3}
Let $\phi\in\sum_{8g}$ be a filling permutation. There exists a minimally intersecting separating filling pair $(\alpha,\beta)$ such that $\phi=\phi^{(\alpha,\beta)}$.
\end{theorem}
\begin{proof}
The existence of a filling pair $(\alpha,\beta)$ on $S_g$ such that $\phi=\phi^{(\alpha,\beta)}$ follows from  the converse part of Nieland’s Theorem 2.1~\cite{nieland2018connected}. Hence, to complete the proof, it only remains to show that  $\alpha$ is separating.

Consider $\alpha=(\alpha_1,\dots,\alpha_{2g})$ and $\beta=(\beta_1,\dots,\beta_{2g})$ as in the proof of Theorem 2.1~\cite{nieland2018connected}. An element $i\in\{1,\dots,8g\}$ corresponds to the $i$-th element of 
$$A_g=\left\{\alpha_1,\beta_1,\alpha_2,\beta_2,\hdots,\alpha_{2g},\beta_{2g},\alpha_1^{-1},\beta_1^{-1},\hdots,\alpha_{2g}^{-1},\beta_{2g}^{-1}\right\}.$$
By Proposition \ref{prop_of_phi_alpha_beta}, the complement $S_g\setminus(\alpha\cup \beta)$ is a disjoint union of two polygons, all sides labelled by $\alpha_i$ are in one polygon and those labelled by $\alpha_i^{-1}$ are in another. Furthermore, for every $i$, the edges labelled by  $\beta_i$ and $\beta_i^{-1}$ are in the same polygon. Now, it follows that $\beta_i$ and $\beta_i^{-1}$ are in the same column of the normal matrix $M(\alpha,\beta)$. Hence, Theorem \ref{classif_of_std_cycle_in_fat_graph} implies that $\alpha$ is separating.
\end{proof}
Next, we study a necessary and sufficient condition for two minimally intersecting separating filling pairs to be in the same mapping class group orbit. The following argument is analogous to Aougab-Huang Lemma 2.3 \cite{aougab2015minimally}.

\begin{lemma}\label{lemma11.3}
Suppose $\Gamma=(\alpha,\beta)$ and $\tilde{\Gamma}=(\tilde{\alpha},\tilde{\beta})$ are two minimally intersecting separating filling pairs on $S_g$ in the same $\M$-orbit. Then $\phi^{(\alpha,\beta)}=\phi^{(\tilde{\alpha},\tilde{\beta})}$ modulo conjugation by permutations of the form 
\begin{center}
    $\mu_g^l\kappa_g^k\delta_g^j\eta_g^i,\hspace{.5 cm} l,i\in\{0,1\};\hspace{.4 cm} j,k\in\{0,1,\hdots,2g-1\},$ where 
\end{center}
\begin{align*}
    \mu_g&=(2,4g+2)(4,4g+4)\hdots(4g,8g),\\
    \kappa_g&=(1,3,\hdots,4g-1)(4g+1,4g+3,\hdots,8g-1),\\
    \delta_g&=(2,4,\hdots,4g)(4g+2,4g+4,\hdots,8g) \text{ and}\\
    \eta_g&=(1,4g+1)(3,4g+3)\hdots(4g-1,8g-1).
\end{align*}
\end{lemma}
\begin{proof}
We consider labelling on $(\alpha,\beta)$ and $(\tilde{\alpha},\tilde{\beta})$ so that $\alpha=(\alpha_1,\hdots,\alpha_{2g})$, $\beta=(\beta_1,\hdots,\beta_{2g})$, $\tilde{\alpha}=(\tilde{\alpha}_1,\hdots,\tilde{\alpha}_{2g})$ and $\tilde{\beta}=(\tilde{\beta}_1,\hdots,\tilde{\beta}_{2g})$ (see Section \ref{sec:5.1}). The given filling pairs are in the same mapping class group orbit implies that there exists $f\in \M$ such that $f(\alpha,\beta)=(\tilde{\alpha},\tilde{\beta})$. Now, there are following four cases to consider.

\noindent \textbf{Case 1.} Consider $f(\alpha_n)=\tilde{\alpha}_{n+k_0}$, for some $0\leq k_0\leq2g-1$ and $f(\beta_m)=\tilde{\beta}_m$, for all $m,n=1,\dots,2g$. In this case, $\phi^{(\alpha,\beta)}$ and $\phi^{(\tilde{\alpha},\tilde{\beta})}$ are conjugate by $\kappa_g^{k_0}$.

\noindent \textbf{Case 2.} Consider  $f(\beta_m)=\tilde{\beta}_{m+j_0}$, for some $0\leq j_0\leq2g-1$ and $f(\alpha_n)=\tilde{\alpha}_n$, for all $m,n=1,\dots,2g$. Then $\phi^{(\alpha,\beta)}$ and
$\phi^{(\tilde{\alpha},\tilde{\beta})}$ are conjugate by $\delta_g^{j_0}$.

\noindent \textbf{Case 3.} Let $f$ preserves the orientation of $\alpha$ but reverses that of $\beta$. Then $\phi^{(\alpha,\beta)}$ and
$\phi^{(\tilde{\alpha},\tilde{\beta})}$ are conjugate by $\mu_g$.

\noindent \textbf{Case 4.} Let $f$ preserves the orientation of $\beta$ but reverses that of $\alpha$. Then $\phi^{(\alpha,\beta)}$ and
$\phi^{(\tilde{\alpha},\tilde{\beta})}$ are conjugate by $\eta_g$.

The action of $f$ lifts to an action by rotation or by exchange on the $4g$-gons so that the original cycles of the filling permutation are fixed unless one of Cases 1-4 applies. Therefore, in the former case, $\phi^{(\alpha,\beta)}$ and
$\phi^{(\tilde{\alpha},\tilde{\beta})}$ are same and in the latter case, they are conjugate by $\mu_g^l\kappa_g^k\delta_g^j\eta_g^i, \text{ where } l,i\in\{0,1\}$ and  $j,k\in\{0,1,\hdots,2g-1\}.$
\end{proof}

\noindent\textbf{Note.} We call a permutation of the form $\mu_g^l\kappa_g^k\delta_g^j\eta_g^i$ as a\emph{ twisting permutation}.
\end{subsection}
\begin{subsection}{A lower bound of \texorpdfstring{$N(g)$}{TEXT}.}
In this section, we  prove the inequality
$f(g)\leq N(g)$, where 
$$f(g)= \frac{\prod\limits_{k=0}^{\frac{g-4}{2}-1}(8+3k)}{4\times (2g)^2\times (\frac{g-4}{2})!}.$$
To prove the inequality, we use similar idea as in the proof of Theorem 1.2~\cite{aougab2015minimally} due to Aougab-Huang, but we are required to adapt it to the differing combinatorics of new separating filling pair $Z$-piece.
Recall that in the proof of Theorem~\ref{theorem:2} (see Section~\ref{proof_th_2}), to construct a minimal filling pair on the surface $S_{g+2}$ from that of $S_g$, we have used a copy of a $S_{2,1}$.
 We label the sub-arcs corresponding to $\alpha'$ and $\beta'$ by $x_1,\hdots,x_6$ and  $y_1,\hdots,y_6$, respectively. We call this as $Z$-piece on which the points below are distinct:

\begin{enumerate}
    \item initial point of $y_1$,
    \item initial point of $x_1$,
    \item end point of $y_6$ and
    \item end point of $x_6$.
\end{enumerate}
For a Z-piece $Z_1$, we define $Z_1^{(\alpha)}$ to be the interior arcs of the separating arc of the Z-piece. Here $Z_1^{(\alpha)}=\{x_2,\hdots,x_5\}$. Similarly we define $Z_1^{(\beta)}=\{y_2,\hdots,y_5\}$. For two Z-pieces $Z_1 $ and $Z_2$, we define $(Z_1,Z_2)^\cap=(Z_1^{(\alpha)}\cap Z_2^{(\alpha)})\cup(Z_1^{(\beta)}\cap Z_2^{(\beta)})$.
By $x_k(Z_1)$, we mean the end point of the arc $x_k$ in the Z-piece $Z_1$.

\begin{lemma}\label{lemma_equality_of_Z_piece}
    If $Z_1$ and $Z_2$ are two Z-pieces on $(\alpha,\beta)$ with  $(Z_1^{(\alpha)}\cap Z_2^{(\alpha)})\neq\emptyset$, then $Z_1=Z_2$.
\end{lemma}
\begin{proof}
    Assume $Z_1$ starts before $Z_2$.
    
    \noindent \textbf{Case 1:} Consider $x_6(Z_1)=x_3(Z_2)$. Then $x_4*x_5*y_4^{-1}$ is a loop implies $x_1*x_2*y_5^{-1}$ must be a loop which is not.
    
    \noindent\textbf{Case 2:} Consider $x_6(Z_1)=x_2(Z_2)$. Then $x_4*x_5*y_4^{-1}$ is a loop implies the concatenation of the $x$-arc before $x_1,x_1$ and an $y$-arc incident at the endpoint of $x_1$ must be a loop, which is impossible.
    
    \noindent\textbf{Case 3:} $x_6(Z_1)=x_4(Z_2)$. Then $x_2*x_3*x_4*y_2$ is a loop  which implies $x_4*x_5*x_6*\tilde{y}$ is a loop, where $\tilde{y}$ is an $y$-arc incident at the endpoint of $x_6$. But this is not possible.
    
    \noindent\textbf{Case 4:} $x_6(Z_1)=x_5(Z_2)$. Then $x_2*x_3*y_3^{-1}$ is a loop which implies $x_1*x_2*y'$ is a loop, where $y'$ is an $y$-edge incident at the end point of $x_2$. But this is not true.
    
    Therefore, $x_6(Z_1)=x_6(Z_2)$ and this implies $Z_1=Z_2$. 
\end{proof}
Now, we are ready to find a lower bound for $N(g)$. For $g\geq 6$, we define
\begin{align*}
    \mathscr{L}_g&=\left\{\left(s_1,\dots,s_\frac{g-4}{2}\right)| s_i\in \mathbb{N}, s_i\leq4(i+1), i=1,\dots,\frac{g-4}{2}\text{ and } s_1<\dots<s_\frac{g-4}{2} \right\}\text{ and}\\
    \mathscr{O}_g&=\left\{(\alpha,\beta)| (\alpha,\beta)\text{ is  an oriented minimal separating filling pair on } S_g \right\}.
\end{align*}
\begin{lemma}\label{lemma5.6}
    $|\mathscr{L}_g|\leq|\mathscr{O}_g|$.
\end{lemma}
\begin{proof}
    To prove the lemma, we construct an injective function from $\mathscr{L}_g$ to $\mathscr{O}_g$. Let 
    $$\left(s_1, \dots, s_\frac{g-4}{2} \right) \in \mathscr{L}_g.$$ 
    Consider an oriented minimal separating filling pair $(\alpha(4),\beta(4))$ on $S_4$ and label the vertices along $\alpha(4)$ by $v_1,v_2\dots,v_8$. At the vertex $v_{s_1}$, we extract an open disk and attach a $Z$-piece to obtain an oriented minimal filling pair $(\alpha(6),\beta(6))$ on $S_{6}$ whose vertices are labelled as follows: the labelling of the vertices $v_i,i<s_1$ are kept same, the vertices on the $Z$-piece are labelled as $v_{s_1},\dots,v_{s_1+4}$ and the subsequent vertices are labelled as $v_{s_1+5},\dots,v_{12}$.
    
    Next, we choose the vertex $v_{s_2}$ and perform a similar operation as above to obtain an oriented minimal filling pair $(\alpha(8),\beta(8))$ on $S_8$. We repeat this process to get an oriented minimal filling pair $(\alpha(g),\beta(g))$ on $S_g$.
    Now we define $f:\mathscr{L}_g\longrightarrow\mathscr{O}_g$ by 
        $$f(s_1,\dots,s_\frac{g-4}{2})=(\alpha(g),\beta(g)).$$
    
    We prove that $f$ is injective. Let  $f(s_1,\dots,s_\frac{g-4}{2})=f(s_1',\dots,s_\frac{g-4}{2}')=(\alpha(g),\beta(g)).$ First, we show that $s_\frac{g-4}{2}=s_\frac{g-4}{2}'$. If not, suppose, without loss of generality, that $s_\frac{g-4}{2}< s_\frac{g-4}{2}'$. By the construction of the surface $S_g$, there are two distinct $Z$-pieces at $s_\frac{g-4}{2}$$^{th}$ and $ s_\frac{g-4}{2}'$$^{th}$ vertices and by Lemma \ref{lemma_equality_of_Z_piece}, they are disjoint. Now, we remove the $Z$-piece at $s_\frac{g-4}{2}$$^{th}$ vertex and attach a disk to obtain an oriented minimal filling pair $(\alpha(g-2),\beta(g-2))$ on $S_{g-2}$. We do this for the vertices corresponding to $s_{\frac{g-4}{2}-1},\dots,s_1$ and finally we get a minimal filling pair $(\alpha(4),\beta(4))$ on $S_4$ and it contains a Z-piece, which is not possible. So $s_\frac{g-4}{2}=s_\frac{g-4}{2}'$. A similar argument shows that $s_i=s_i^{'} \text{ for all, } i=1,\dots, \frac{g-4}{2}-1$. Hence, $f$ is injective.
\end{proof}
\begin{corollary}
    $N(g)\geq f(g).$
\end{corollary}
\begin{proof}
    Consider pairwise distinct natural numbers $t_1,\dots,t_\frac{g-4}{2}$ satisfying $t_k\leq4(k+1),$ for  $k=1,\dots,\frac{g-4}{2}$. Then there exists a permutation $\sigma$ such that $\left(t_{\sigma(1)},\dots,t_{\sigma(\frac{g-4}{2})}\right)\in\mathscr{L}_g$. Furthermore, by consideration we have $3k+5$ many choices for each $t_k$  which implies
    $$|\mathcal{L}_g|\geq\frac{1}{(\frac{g-4}{2})!}\times \prod\limits_{k=1}^{\frac{g-4}{2}}(3k+5).$$ 
    Therefore, the lower bound is obtained by dividing by maximum number of twisting permutations,
    $$N(g)\geq\frac{|\mathcal{O}_g|}{4 \times (2g)^2}\geq\frac{|\mathcal{L}_g|}{4 \times (2g)^2}\geq\frac{1}{4\times (2g)^2\times (\frac{g-4}{2})!}\times\prod\limits_{k=1}^{\frac{g-4}{2}}(3k+5).$$
\end{proof}
\end{subsection}

\begin{lemma}\label{gr}
    $f(g)\sim\frac{3^{\frac{g}{2}}}{g^{\frac{1}{3}}}.$
\end{lemma}
In Lemma~\ref{gr}, $\sim$ denotes the growth-rate relation $$r(g)\sim s(g)\iff 0< \lim_{g\to\infty} \frac{r(g)}{s(g)}<\infty.$$
\begin{proof}[Proof of Lemma~\ref{gr}]
     \begin{align*}
         \frac{f(g)}{3^{\frac{g}{2}}/g^{\frac{1}{3}}}=&\frac{3^\frac{g-4}{2}\times g^\frac{1}{3}}{16\times g^2\times 3^\frac{g}{2}} \cdot \frac{\left(\frac{5}{3}+1\right) \left(\frac{5}{3}+2\right)  \cdots \left(\frac{5}{3}+\frac{g-4}{2}\right)}{\left(\frac{g-4}{2}\right)!}\\
         =& \frac{3}{16\times 3^2\times 5} \cdot \frac{\left(\frac{g-4}{2}\right)^\frac{5}{3}}{g^\frac{5}{3}}\cdot \frac{\frac{5}{3}\left(\frac{5}{3}+1\right) \left(\frac{5}{3}+2\right)  \cdots \left(\frac{5}{3}+\frac{g-4}{2}\right)}{\left(\frac{g-4}{2}\right)! \left(\frac{g-4}{2}\right)^\frac{5}{3}}
     \end{align*}
     Now, $\frac{\left(\frac{g-4}{2}\right)^\frac{5}{3}}{g^\frac{5}{3}} \to \left( \frac{1}{2}\right)^\frac{5}{3}$ and  $\frac{\frac{5}{3}\left(\frac{5}{3}+1\right) \left(\frac{5}{3}+2\right)  \cdots \left(\frac{5}{3}+\frac{g-4}{2}\right)}{\left(\frac{g-4}{2}\right)! \left(\frac{g-4}{2}\right)^\frac{5}{3}}\to \frac{1}{\Gamma\left(\frac{5}{3}\right)}$   as  $g \to \infty$ (see Section 13.1.3 \cite{krantz1999handbook}), where $\Gamma$, the well known gamma function, is defined by 
     $$\Gamma(z)=\int_{0}^{\infty}t^{z-1} e^{-t}\,dt,$$
     for all $z\in \C$ with $Re(z)>0$.
     
     Therefore,
     $$\frac{f(g)}{3^{\frac{g}{2}}/g^{\frac{1}{3}}} \to \frac{3}{16\times 3^2\times 5} \cdot \left( \frac{1}{2}\right)^\frac{5}{3} \cdot\frac{1}{\Gamma\left(\frac{5}{3}\right)} \text{ as } g \to \infty.$$
\end{proof}

\end{section}

\begin{section}{Upper bound of \texorpdfstring{$N(g)$}{N(g)}}
In this section, we prove the inequality $N(g)\leq 2(2g-2)(2g-2)!$ in the following proposition which concludes the proof of Theorem~\ref{thm:1.3}. The arguments of the proof are similar to those of Aougab-Huang (Section 2.3 \cite{aougab2015minimally}).
\begin{proposition}\label{prop:6.1}
    For $g\geq4$, we have $N(g)\leq 2(2g-2)(2g-2)!$.
\end{proposition}
Before going into the proof of Proposition \ref{prop:6.1}, we prove a technical result in Lemma \ref{lemma:6.1} which is essential for the proof of the proposition.

\begin{lemma}\label{lemma:6.1}
    Let $\phi\in\Sigma_{8g}$ be a filling permutation. Then $\phi$ is of the form $Q^{4g}C$, where C is a square root of $Q^{4g}\tau$. Conversely, if $C$ is a square root of $Q^{4g}\tau$ with $\phi=Q^{4g}C$ satisfying conditions $(1)-(5)$ of Proposition~\ref{prop_of_phi_alpha_beta}, then $\phi$ is a filling permutation.
\end{lemma}
\begin{proof}
    Let $\phi$ be a filling permutation. Then,
    $$\left(Q^{4g}\phi\right)^2=Q^{4g}\left(\phi Q^{4g}\phi\right)= Q^{4g}\tau.$$
    Conversely, 
    $$\phi Q^{4g}\phi=Q^{4g}C Q^{4g}Q^{4g}C=Q^{4g}C^2=\tau.$$
\end{proof}
\begin{proof}[Proof of Proposition \ref{prop:6.1}]
    In light of Lemma \ref{lemma:6.1}, to prove the proposition, we find an upper bound of the number of square roots of the permutation $Q^{4g}\tau$. We have 
    \begin{align*}
        Q^{4g}\tau=&(1,4g+3)(3,4g+5)\dots(4g-3,8g-1)(4g-1,4g+1)\\ &(2,4g+4)(4,4g +6)\dots(4g-2,8g)(4g,4g+2).
    \end{align*}
    A square root $C$ of $Q^{4g}\tau$ is obtained by taking the transpositions of $Q^{4g}\tau$ in pairs and interleaving them. For instance, for the pair of transpositions $(1,4g+3)$ and $(2,4g+4)$, the possible arrangements are $(1,2,4g+3,4g+4)$ and $(1,4g+4,2,4g+3)$. In each pair of transpositions, one comes from $\{(1,4g+3),(3,4g+5),\dots,(4g-1,4g+1)\}$ and the other from $\{(2,4g+4),(4,4g+6),\dots,(4g-2,8g)\}$, as $\phi=Q^{4g}C$ is parity respecting and sends even numbers to odds and $Q^{4g}$ sends odds to odds and evens to evens.
    
    Suppose, for a suitable choice of $C$, the permutation $\phi$ satisfies conditions $(1)-(5)$ of Proposition \ref{prop_of_phi_alpha_beta} and $\phi=\phi_1\phi_2$, where $\phi_1$ and $\phi_2$ are disjoint $4g$-cycles. Now, we have the following cases to consider.
    
    \noindent\textbf{Case 1:} The odd entries of $\phi_1$ come from $\{1,3,\hdots,4g-1\}$ and the even entries of $\phi_1$ are of the form $4k$ (following properties $(3)$ and $(5)$ of proposition \ref{prop_of_phi_alpha_beta}). In this case, we have a unique choice for the $4$-cycle. More precisely, for the transpositions $(i,i+4g+2)$  and $(j,j+4g+2)$, the only possible $4$-cycle is $(i,j,i+4g+2,j+4g+2)$, where $i\leq4g$ is odd, $j\leq8g$ is of the form $4k$ and $(i,j)\neq(4g-1,4g)$. Here, $i+4g+2$ and $j+4g+2$ are considered with modulo ($8g$). Furthermore, for the transpositions $(4g-1, 4g+1)$  and $(4g, 4g+2)$, the only possible $4$-cycle is $(4g-1,4g,4g+1,4g+2)$. Such pairs can be chosen in $(2g)!$ ways implies the number of such square roots $C$ is bounded above by $(2g)!$.
    
    \noindent\textbf{Case 2:} The odd entries of $\phi_1$ comes from $\{1,3,\dots,4g-1\}$ and the even entries of $\phi_1$ are of the form $4k+2$. A similar argument shows that in this case also there are at most $(2g)!$ choices.
    
    Therefore, the square root $C$ and hence $\phi$ has at most $2\times(2g)!$ choices. Again using the conditions of Proposition \ref{prop_of_phi_alpha_beta}, we further rule out some more choices of $\phi$. Here, we find those $\phi$'s which satisfy $\phi^2(1)=1$ and subtract this from the whole. First, take the transpositions $(1,4g+3)$ and $(i,j)$, where
    \[(i,j)= \begin{cases} 
          (i, i+4g+2), & \text{ if } i< 4g \text{ even and } \\
          (4g, 4g+2), & \text{ if } i=4g.
       \end{cases}
    \]
    Then $(1,i,4g+3,j)$ is a choice for a cycle of $C$. Next, pair the transpositions $(i-2,4g+i)$ and $(4g-1,4g+1)$ to get another $4$-cycle $(4g-1,4g+i,4g+1,i-2)$ of $C$ (note that when $i=2$, $(i-2)$ is replaced by $4g$) and choose the other $4$-cycles as discussed in Case 1 and Case 2. As $\phi^2(1)=1,$ and so $\phi$ is not a product of two $4g$-cycles, the number of filling permutations $\phi$ is bounded above by $2(2g)!-4g\times(2g-2)!=4g(2g-2)\times(2g-2)!$.
    
     Two minimally intersecting filling pairs are in the same $\M$ orbit if and only if the corresponding filling permutations are conjugate by twisting permutations $\mu_g^l\kappa_g^k\delta_g^j\eta_g^i$, where $l,i\in\{0,1\},j,k\in\{0,1,\dots,2g-1\}$ and $\mu_g,\kappa_g,\delta_g,\eta_g$ are defined as in Lemma \ref{lemma11.3}. As there are at least $2g$ many twisting permutations, the number of $\M$ orbits of  minimally intersecting filling pair is bounded by $\frac{4g(2g-2)(2g-2)!}{2g}=2(2g-2)\times(2g-2)!.$
\end{proof}

\begin{remark}
   The upper bound for non-separating filling pairs given by Aougab-Huang in Theorem 1.1~\cite{aougab2015minimally} is $2^{(2g-2)}(4g -5) \times(2g-3)!$  which is strictly larger than the separating case exponentially.
\end{remark}
\end{section}
\begin{section}{The length function}
In this section, we prove Theorem \ref{theorem_lenth_function}. First, we study a technical lemma which deals with injectivity radius (see Section $4.1$ in \cite{buser2010geometry}) and length of fillings. The injectivity radius of a hyperbolic surface $X$ is denoted by $r_{\mathrm{inj}}(X)$. We note that the lemma is essential in proving the length function $\mathcal{F}_g$ is proper.
\begin{lemma}\label{lemma_injectivity_radius}
    Given $M\in \R_{>0}$, there exists an $\epsilon>0$ such that $\mathcal{F}_g(X)\geq M$ if $r_{\mathrm{inj}}(X)\leq \epsilon$.
\end{lemma}
\begin{proof}
    The proof follows from a similar argument as in the proof of Theorem $1.3$ in \cite{aougab2015minimally}.
\end{proof}
\begin{lemma}\label{lemma7.2}
$\mathcal{F}_g$ is proper.
\end{lemma}
\begin{proof}
    The proof follows from a similar argument as in the proof of Theorem $1.3$ in \cite{aougab2015minimally}.
\end{proof}
\begin{lemma}\label{lemma7.3}
    The function $\mathcal{F}_g$ is a topological Morse function.
\end{lemma}
\begin{proof}
    The proof follows from a similar argument as in the proof of Theorem $1.3$ in \cite{aougab2015minimally}.
\end{proof}
\begin{lemma}\label{lemma7.4}
    $\mathcal{F}_g\geq m_g$, where $m_g$ is the perimeter of the regular right-angled hyperbolic $4g$-gon.
\end{lemma}
\begin{proof}
    Let the minimum value of $\mathcal{F}_g$ occurs at $X\in \mathcal{M}_g$ and $(\alpha,\beta)$ be a shortest length filling pair on $X$, where both of $\alpha$ and $\beta$ are geodesics. We cut the surface along $\alpha\cup\beta$ and we obtain two hyperbolic $4g$-gons $P_1$ and $P_2$, each of area $\pi(2g-2)$. It follows from the fact that the hyperbolic $n$-gon with least perimeter is regular (see Bezdek \cite{bezdek1984elementarer}), the polygons $P_1$ and $P_2$ are regular. The Gauss-Bonnet theorem implies that $P_1$ and $P_2$ are regular right-angled $4g$-gons. Hence, the proof follows.
\end{proof}
Next, we find the growth of the set $\mathcal{B}_g$ in the following lemma.
\begin{lemma} \label{lemma7.5}
    For $g\geq4,$ we have $|\mathcal{B}_g|=N(g)$.
\end{lemma}
\begin{proof}
    By Lemma~\ref{lemma7.2}, $\mathcal{F}_g$ is proper and hence $\mathcal{B}_g$ is compact. All the points of $\mathcal{B}_g$ are isolated as they are critical points of the Morse function $\mathcal{F}_g$. Therefore, $\mathcal{B}_g$ is a finite set.
    
    Now, we find a one to one correspondence between the set $\mathcal{B}_g$ and the collection of $\M$-orbits of minimal separating filling pairs. For each of $\M$-orbits, we associate an element $X$ of $\mathcal{B}_g$ by simply identifying the edges of two disjoint regular right-angled $4g$-gons in accordance with the chosen orbit. It is straightforward to see that this association is surjective. The injectivity follows from Proposition~\ref{lemma:7.4}. This completes the proof.
\end{proof}
\begin{proposition} \label{lemma:7.4}
Let $(\alpha,\beta)$ and $(\tilde\alpha,\tilde\beta)$ be two minimal filling pairs on a hyperbolic surface $X\in \mathcal{B}_g$ each of length $m_g$. Then there exists a homeomorphism $\phi:X\longrightarrow X$ such that $\phi(\alpha,\beta)=(\tilde\alpha,\tilde\beta)$.
\end{proposition}
Before going to the proof, we recall the notion of equivariant tiling, chamber system and Delaney-Dress symbols and their equivalence which are essential for the proof of Proposition~\ref{lemma:7.4}. An \emph{equivariant tiling} is a pair $(\mathcal{T},\Gamma)$, where $\mathcal{T}$ is a tiling of $\mathbb{H}$ and $\Gamma$ is a discrete subgroup of $\mathrm{PSL}(2, \mathbb{R})$ satisfying $\gamma\mathcal{T}=\{\gamma A|A\in \mathcal{T}\}=\mathcal{T}$ for all $\gamma\in \Gamma$. Two equivariant tilings $(\mathcal{T},\Gamma)$ and $(\mathcal{T}',\Gamma')$ are said to be equivariantly equivalent if there exists a homeomorphism $\phi:\mathbb{H}\to \mathbb{H}$ such that $\phi \mathcal{T}=\mathcal{T}'$ and $\Gamma'= \phi \Gamma \phi^{-1}$.

To every vertex, edge and tile of $\mathcal{T}$ choose an interior point, called a $0$-, $1$- and $2$-center respectively. For every tile $A\in\mathcal{T}$, join its $2$-center with $0$- and $1$-centers by geodesics (see Figure~\ref{Fig:Chamber}) which decompose the tiling $A$ into triangles. We call each triangle a \emph{chamber} and the set of all triangles, denoted by $\mathcal{C}_\mathcal{T}$, is called \emph{chamber system}. Consider $\mathscr{D}=\mathcal{C}_\mathcal{T}/\Gamma$. The edge opposite the $i$-center of a chamber is called an $i$-edge. For $T\in \mathcal{C}_\mathcal{T}$, the neighbour triangle that share the $i$-edge of $T$, is denoted by $\sigma_i(T)$. Thus we have functions $\sigma_i:\mathcal{C}_\mathcal{T}\to \mathcal{C}_\mathcal{T}$ satisfying $\sigma_i(\gamma T)=\gamma \sigma_i(T)$, for $i=0,1,2$, $T\in \mathcal{C}_\mathcal{T}$ and $\gamma \in \Gamma$. This induces functions $\sigma_i^*:\mathscr{D}\longrightarrow\mathscr{D}$ and let $\mathscr{E}\coloneqq \{(\{D,D'\},i)|D,D'\in \mathscr{D}\text{ and }D\sigma^*_i=D'\}$.
For $0\leq i<j \leq2$, we define  
 $m_{ij}:\mathscr{D}\longrightarrow\mathbb{N}$ with
$$m_{ij}(D)=\min \left\{m\in \mathbb{N}|C(\sigma_i\sigma_j)^m=C \text{ for all } C\in D\right\}.$$
Then the system $(\mathscr{D},m):=((\mathscr{D},\mathscr{E});m_{0,1},m_{0,2},m_{1,2})$ is called a Delaney-Dress symbol if the following hold $\forall D\in \mathscr{D}$ and $0\leq i<j\leq 2$:
\begin{description}
    \item[DS1] $m_{ij}(D)=m_{ij}(D\sigma_i)=m_{ij}(D\sigma_j)$
    \item[DS2] $D(\sigma_i\sigma_j)^{m_{ij}(D)}=D(\sigma_j\sigma_i)^{m_{ij}(D)}=D$
    \item[DS3] $m_{0,2}(D)=2$
    \item[DS4] $m_{0,1}\geq2$
    \item[DS5] $m_{1,2}(D)\geq3.$
\end{description}
Two Delaney-Dress symbols $(\mathscr{D},m)$ and $(\mathscr{D'},m')$ are called isomorphic if
and only if there exists a bijection $\pi:\mathscr{D}
\longrightarrow\mathscr{D}'$ with $(\pi D)\sigma_k=\pi(D\sigma_k)$ and $m_{ij}'(\pi D)=m_{ij}(D)$ for all $D\in \mathscr{D}, 0\leq k\leq2$ and $0\leq i<j \leq 2$. For more details, we refer the reader to Section 1  \cite{huson1993generation}.

\begin{figure}
\begin{center}
\begin{tikzpicture}
\draw ({2*cos(45)},{2*sin(45)}) to[bend left=20] ({2*cos(90)},{2*sin(90)}) to[bend left=20] ({2*cos(135)},{2*sin(135)}) to[bend left=20] ({2*cos(180)},{2*sin(180)}) to[bend left=20] ({2*cos(225)},{2*sin(225)})to[bend left=20] ({2*cos(270)},{2*sin(270)}) to[bend left=20] ({2*cos(315)},{2*sin(315)}) to[bend left=20] ({2*cos(360)},{2*sin(360)}) to[bend left=20] ({2*cos(45)},{2*sin(45)});

\draw ({5+2*cos(45)},{2*sin(45)}) to[bend left=20] ({5+2*cos(90)},{2*sin(90)}) to[bend left=20] ({5+2*cos(135)},{2*sin(135)}) to[bend left=20] ({5+2*cos(180)},{2*sin(180)}) to[bend left=20] ({5+2*cos(225)},{2*sin(225)})to[bend left=20] ({5+2*cos(270)},{2*sin(270)}) to[bend left=20] ({5+2*cos(315)},{2*sin(315)}) to[bend left=20] ({5+2*cos(360)},{2*sin(360)}) to[bend left=20] ({5+2*cos(45)},{2*sin(45)});

\draw [thin,cyan] (5,0)--({5+2*cos(45)},{2*sin(45)});
\draw [thin,cyan] (5,0)--({5+2*cos(90)},{2*sin(90)});
\draw [thin,cyan] (5,0)--({5+2*cos(135)},{2*sin(135)});
\draw [thin,cyan] (5,0)--({5+2*cos(180)},{2*sin(180)});
\draw [thin,cyan] (5,0)--({5+2*cos(225)},{2*sin(225)});
\draw [thin,cyan] (5,0)--({5+2*cos(270)},{2*sin(270)});
\draw [thin,cyan] (5,0)--({5+2*cos(315)},{2*sin(315)});
\draw [thin,cyan] (5,0)--({5+2*cos(360)},{2*sin(360)});

\draw [thin,blue] (5,0)--({5+1.7*cos(22.5)},{1.7*sin(22.5)});
\draw [thin,blue] (5,0)--({5+1.7*cos(90-22.5)},{1.7*sin(90-22.5)});
\draw [thin,blue] (5,0)--({5+1.7*cos(135-22.5)},{1.7*sin(135-22.5)});
\draw [thin,blue] (5,0)--({5+1.7*cos(180-22.5)},{1.7*sin(180-22.5)});
\draw [thin,blue] (5,0)--({5+1.7*cos(225-22.5)},{1.7*sin(225-22.5)});
\draw [thin,blue] (5,0)--({5+1.7*cos(270-22.5)},{1.7*sin(270-22.5)});
\draw [thin,blue] (5,0)--({5+1.7*cos(315-22.5)},{1.7*sin(315-22.5)});
\draw [thin,blue] (5,0)--({5+1.7*cos(360-22.5)},{1.7*sin(360-22.5)});

\draw[thick, green] ({5+1.7*cos(22.5)},{1.7*sin(22.5)}) node{\tiny$\bullet$} ({5+1.7*cos(90-22.5)},{1.7*sin(90-22.5)}) node{\tiny$\bullet$}
({5+1.7*cos(135-22.5)},{1.7*sin(135-22.5)}) node{\tiny$\bullet$}
({5+1.7*cos(180-22.5)},{1.7*sin(180-22.5)}) node{\tiny$\bullet$}
({5+1.7*cos(225-22.5)},{1.7*sin(225-22.5)}) node{\tiny$\bullet$}
({5+1.7*cos(270-22.5)},{1.7*sin(270-22.5)}) node{\tiny$\bullet$}
({5+1.7*cos(315-22.5)},{1.7*sin(315-22.5)}) node{\tiny$\bullet$}
({5+1.7*cos(360-22.5)},{1.7*sin(360-22.5)}) node{\tiny$\bullet$};

\draw[red] ({5+2*cos(45)},{2*sin(45)}) node{\tiny$\bullet$} ({5+2*cos(90)},{2*sin(90)}) node{\tiny$\bullet$}
({5+2*cos(135)},{2*sin(135)}) node{\tiny$\bullet$}
({5+2*cos(180)},{2*sin(180)}) node{\tiny$\bullet$}
({5+2*cos(225)},{2*sin(225)}) node{\tiny$\bullet$}
({5+2*cos(270)},{2*sin(270)}) node{\tiny$\bullet$}
({5+2*cos(315)},{2*sin(315)}) node{\tiny$\bullet$}
({5+2*cos(360)},{2*sin(360)}) node{\tiny$\bullet$};

\draw (5,0) node{$\bullet$};

\draw (2.5,-3) node{\textcolor{red}{$\bullet$} : $0$-center} (2.5,-3.5) node{\textcolor{green}{$\bullet$} : $1$-center} (2.5,-4) node{$\bullet$ : $2$-center};
\end{tikzpicture}  
\end{center}
\caption{A tile and its chambers.}
\label{Fig:Chamber}
\end{figure}
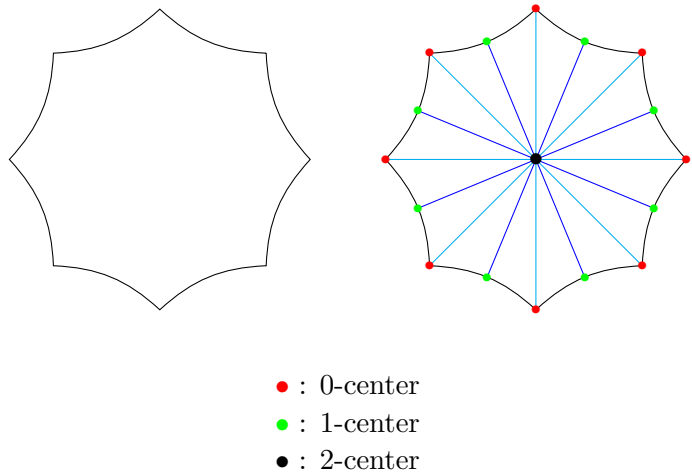

Next, we state a lemma (for details, see Lemma 1.1 \cite{huson1993generation}) which will be used to prove Proposition \ref{lemma:7.4}.
\begin{lemma}\label{lemma0.2}
Two equivariant tilings $(\mathcal{T},\Gamma)$ and $(\mathcal{T}',\Gamma')$ are (equivariantly)
equivalent if and only if the corresponding Delaney-Dress symbols $(\mathscr{D},m)$ and $(\mathscr{D'},m')$ are isomorphic. 
\end{lemma}

\begin{proof}[Proof of Proposition \ref{lemma:7.4}]
Consider the filling pairs $(\alpha,\beta)$ and $(\tilde\alpha,\tilde\beta)$ as in the proposition. Let $\Gamma$ be the discrete subgroup of $\mathrm{PSL}_2(\mathbb{R})$ such that $X\cong \mathbb{H}/\Gamma$. The lifts $\mathscr{T}$ and $\mathscr{T}'$ of $(\alpha,\beta)$ and $(\tilde\alpha,\tilde\beta)$ respectively are two tilings of $\mathbb{H}$ by regular right-angled $4g$-gons. Thus we have two equivariant tilings $(\mathcal{T},\Gamma)$ and $(\mathcal{T}',\Gamma)$.

Consider the chamber systems $\mathcal{C}_\mathcal{T}$ and $\mathcal{C}_{\mathcal{T}'}$ . Let $(\mathscr{D},m)$ and $(\mathscr{D'},m')$ be the corresponding Delaney-Dress symbols. Then,
\begin{align*}
    & m_{0,2}(D)=2= m_{0,2}'(D'),\\
    &m_{1,2}(D)=4=m_{1,2}'(D') \text{ and}\\
    &m_{0,1}(D)=4g=m_{0,1}'(D'), \text{ for all } D\in \mathscr{D}, D'\in \mathscr{D'}. 
\end{align*}
As $|\mathscr{D}|=|\mathscr{D'}|=16g$ and value of $m$ and $m'$ are identical,  the Delaney-Dress symbols $(\mathscr{D},m)$ and $(\mathscr{D'},m')$ are isomorphic and hence by Lemma \ref{lemma0.2}, $(\mathcal{T},\Gamma)$ and $(\mathcal{T}',\Gamma)$ are equivariantly equivalent. Therefore, there exists a homeomorphism $\phi:\mathbb{H}\longrightarrow\mathbb{H}$ such that $\phi\mathscr{T}=\mathscr{T'}$. We can take this homeomorphism $\phi$ as orientation preserving, otherwise we compose it by a suitable reflection. This homeomorphism projects to a homeomorphism $\tilde\phi$ of the surface such that $\tilde\phi(\alpha,\beta)=(\tilde\alpha,\tilde\beta)$.
\end{proof}
\begin{lemma}\label{lem:7.8}
    For $X\in \mathcal{B}_g,g\geq4$, the injectivity radius $r_{\mathrm{inj}}(X)$ of $X$ satisfies
    $$r_{\mathrm{inj}}(X)\geq\frac{1}{2}\times \cosh^{-1}\left[ 8\cos^3\left(\frac{2\pi}{4g}\right)+8\cos^2\left(\frac{2\pi}{4g}\right)-1\right].$$
\end{lemma}
\begin{proof}
    To prove the lemma it suffices to show that for any simple closed geodesic $\gamma$ on $X$, we have 
    $$l_X(\gamma)\geq \cosh^{-1}\left[ 8\cos^3\left(\frac{2\pi}{4g}\right)+8\cos^2\left(\frac{2\pi}{4g}\right)-1\right].$$
    Given a simple closed geodesic $\gamma$ on $X$, let $\Tilde{\gamma}$ be the geodesic segments of $\gamma$ on the two $4g$-gons which are obtained by cutting the surface $X$ along a minimally intersecting filling pair $(\alpha,\beta)$ of length $m_g$. Then $\Tilde{\gamma}$ is a disjoint union of geodesic arcs with end points on the boundary of the polygons. Now, there are two cases to be considered.
    
    \noindent \textbf{Case 1:} $i(\gamma,\alpha)\neq0.$ In this case, $\Tilde{\gamma}$ contains at least two arcs of length at least the length of a side of the $4g$-gons (Figure \ref{fig:7.2}$(i)$). Therefore,
    $$l_X\left(\gamma\right)\geq 2\times\cosh^{-1}\left(2\cos\left(\frac{2\pi}{4g}\right)+1\right), \text{ where } g\geq4.$$

    \noindent \textbf{Case 2:} $i(\gamma,\alpha)=0$. In this case, either $\Tilde{\gamma}$ contains an arc that surpasses at least three consecutive sides of a $4g$-gon (Figure \ref{fig:7.2}$(ii)$) or it contains at least two arcs as discussed in Case 1. In the latter case, the inequality as in Case 1 holds and in the former case, by using basic hyperbolic trigonometry (Theorem 2.4.1 \cite{buser2010geometry}), we have 
    $$l_X(\gamma)\geq \cosh^{-1}\left[ 8\cos^3\left(\frac{2\pi}{4g}\right)+8\cos^2\left(\frac{2\pi}{4g}\right)-1\right].$$
    Now, the proof follows by comparing the above two inequalities.
\end{proof}
\begin{remark}
    The right hand side of the inequality of Lemma \ref{lem:7.8} is strictly increasing with $g$. Hence, we get a universal lower bound
    $$r_{\mathrm{inj}}(X)\geq\frac{1}{2}\times \cosh^{-1}\left[ 8\cos^3\left(\frac{2\pi}{4\times4}\right)+8\cos^2\left(\frac{2\pi}{4\times4}\right)-1\right]\simeq 1.594,$$
    for all $X\in \mathcal{B}_g \text{ and }g\geq4.$
\end{remark}
\begin{remark}
     For every even genus $g\geq4$, there exists a surface $X\in\mathcal{B}_g$ for which equality holds in Lemma \ref{lem:7.8}. In particular, for $g=4$ consider the common perpendicular arc between the edges labelled by $\beta_7$ in the second polygon of Figure \ref{One_of_two_poly}. This arc projects onto a geodesic with length as given in the right hand side of the inequality of Lemma \ref{lem:7.8}. The existence for the higher genus follows from the construction of minimally intersecting separating filling pair.
\end{remark}
\begin{figure}[htbp]
    \centering
    \begin{tikzpicture}
        \draw ({6.5+2*cos(0)},{2*sin(0)}) to [bend right] ({6.5+2*cos(330)},{2*sin(330)});
        \draw ({6.5+2*cos(330)},{2*sin(330)}) to [bend right] ({6.5+2*cos(300)},{2*sin(300)});
        \draw ({6.5+2*cos(300)},{2*sin(300)}) to [bend right] ({6.5+2*cos(270)},{2*sin(270)});
        \draw ({6.5+2*cos(270)},{2*sin(270)}) to [bend right] ({6.5+2*cos(240)},{2*sin(240)});
        \draw ({6.5+2*cos(240)},{2*sin(240)}) to [bend right] ({6.5+2*cos(210)},{2*sin(210)});
        \draw[cyan,thick] ({6.5+1.8*cos(225)},{1.8*sin(225)}) to [bend left=20] ({6.5+1.8*cos(345)},{1.8*sin(345)});

        \draw ({2*cos(157.5)},{2*sin(157.5)}) to [bend left] (-2,0);
        \draw (-2,0) to [bend left] ({2*cos(202.5)},{2*sin(202.5)});
        \draw ({2*cos(202.5)},{2*sin(202.5)}) to [bend left] ({2*cos(225)},{2*sin(225)});
        \draw ({2*cos(225)},{2*sin(225)}) to [bend left] ({2*cos(247.5)},{2*sin(247.5)});
        \draw ({2*cos(247.5)},{2*sin(247.5)}) to [bend left] ({2*cos(270)},{2*sin(270)});
        \draw ({2*cos(270)},{2*sin(270)}) to [bend left] ({2*cos(292.5)},{2*sin(292.5)});
        \draw ({2*cos(292.5)},{2*sin(292.5)}) to [bend left] ({2*cos(315)},{2*sin(315)});
        \draw ({2*cos(315)},{2*sin(315)}) to [bend left] ({2*cos(337.5)},{2*sin(337.5)});
        \draw ({2*cos(337.5)},{2*sin(337.5)}) to [bend left] ({2*cos(360)},{2*sin(360)});

        \draw[thick, cyan] ({1.85*cos(348.75)},{1.85*sin(348.75)}) to [bend right] ({1.85*cos(303.75)},{1.85*sin(303.75)});
        \draw[thick, cyan] ({1.85*cos(168.75)},{1.85*sin(168.75)}) to [bend left] ({1.85*cos(213.75)},{1.85*sin(213.75)});

        \draw (0,-2.7) node {$(i)$} (6.5,-2.7) node{$(ii)$};
    \end{tikzpicture}
    \caption{$(i)$ Two possible arcs of $\Tilde{\gamma}.$  $(ii)$ An arc of $\Tilde{\gamma}$ surpassing $3$ consecutive side.}
    \label{fig:7.2}
\end{figure}
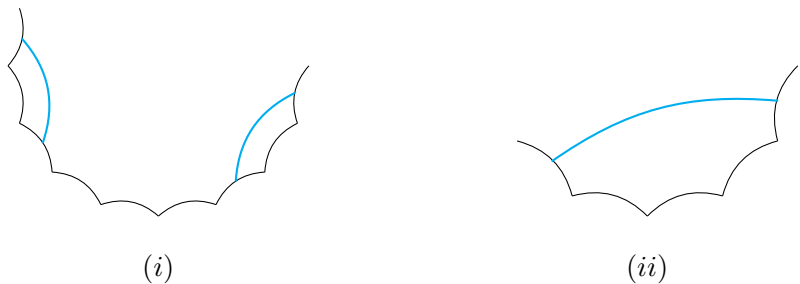

\begin{proof}[Proof of Theorem \ref{theorem_lenth_function}]
By Lemma \ref{lemma7.2}, $\mathcal{F}_g$ is a proper function and that $\mathcal{F}_g$ is a Morse function follows from Lemma \ref{lemma7.3}. The inequality $\mathcal{F}_g\geq m_g$ follows from Lemma \ref{lemma7.4}. It follows from Lemma~\ref{lemma7.5} that $|\mathcal{B}_g|=N(g)$. Finally, the inequality on injectivity radius is proved in Lemma \ref{lem:7.8}.
\end{proof}

\begin{remark}
    The minimal length $l_{\mathrm{ns}}(g)$ for minimally intersecting filling pairs, found by Aougab-Huang (Theorem 1.3 \cite{aougab2015minimally}), is half the perimeter of a regular right-angled $(8g -4)$-gon, that is 
    $$l_{\mathrm{ns}}(g)=(4g-2)\times \cosh ^{-1}\left( 2\left[\cos \left(\frac{2\pi}{8g-4}\right)+\frac{1}{2}\right]\right).$$ 
    But the minimal length $l_\mathrm{s}(g)$ for separating case is $m_g$  which is described in Theorem \ref{theorem_lenth_function}. Therefore, it follows that for $g\geq 4$ even, $l_\mathrm{s}(g)$ is strictly larger than $l_{\mathrm{ns}}(g)$ and the asymptotic difference between these two lengths is $2 \cosh^{-1}(3)$, that is 
    $$\lim_{g\rightarrow\infty}\left[l_\mathrm{s}(g)-l_{\mathrm{ns}}(g)\right]=2 \cosh^{-1}(3).$$
\end{remark}

\end{section}
\section*{Acknowledgements}

The authors would like to thank the anonymous referee for useful suggestions and comments in improving the exposition.

\end{document}